\documentclass[12pt]{amsart}
\usepackage{a4wide}
\usepackage{amssymb}
\usepackage{amsthm}
\usepackage{amsmath}
\usepackage{amscd}
\usepackage{verbatim}
\usepackage{mathrsfs}
\usepackage{dsfont,bbm} 
\usepackage{youngtab}
\usepackage[usenames]{color}
\usepackage{tikz,graphicx}
\usepackage{hyperref}

\numberwithin{equation}{section}
\theoremstyle{plain}
\newtheorem{theorem}{Theorem}[section]
\newtheorem{corollary}[theorem]{Corollary}
\newtheorem{lemma}[theorem]{Lemma}
\newtheorem{proposition}[theorem]{Proposition}

\theoremstyle{definition}

\newtheorem*{example}{Example}
\theoremstyle{remark}
\newtheorem*{remark}{Remark}
\newtheorem*{remarks}{Remarks}
\newtheorem*{Note}{Note added}

\newcommand{\blue}[1]{{\color{blue}  #1}}

\DeclareMathOperator{\lev}{lev}
\DeclareMathOperator{\Den}{Den}
\DeclareMathOperator{\SL}{SL}
\DeclareMathOperator{\GL}{GL}
\DeclareMathOperator{\Gal}{Gal}
\DeclareMathOperator{\lcm}{lcm}
\DeclareMathOperator{\im}{Im}
\DeclareMathOperator{\ord}{ord}
\DeclareMathOperator{\ch}{ch}
\DeclareMathOperator{\mult}{mult}
\newcommand{\eup}{\textrm{e}}
\newcommand{\iup}{\textrm{i}}
\renewcommand{\H}{\mathbb{H}}
\newcommand{\Rat}{\mathbb Q}

\newcommand{\N}{\mathbb N}

\newcommand{\Z}{\mathbb Z}
\newcommand{\NN}{\Z_{+}}

\newcommand{\Symm}{\mathfrak{S}}
\newcommand{\M}{\mathscr{M}}
\newcommand{\abs}[1]{\lvert#1\rvert}
\newcommand{\la}{\lambda}
\newcommand{\La}{\Lambda}
\newcommand{\ip}[2]{\langle#1,#2\rangle}

\newcommand{\A}{\mathrm A}
\newcommand{\B}{\mathrm B}
\newcommand{\C}{\mathrm C}
\newcommand{\D}{\mathrm D}

\newcommand{\qbin}[2]{\genfrac{[}{]}{0pt}{}{#1}{#2}}

\newcommand{\hfrak}{\mathfrak{h}}

\begin{document}

\title[Rogers--Ramanujan identities and their arithmetic properties]
{A framework of Rogers--Ramanujan identities and their arithmetic properties}

\author{Michael J.~Griffin, Ken Ono, and S.~Ole Warnaar}

\address{Department of Mathematics, 
Princeton University, Princeton, NJ 08544, USA}
\email{mjg4@princeton.edu}

\address{Department of Mathematics and Computer Science,
Emory University, Atlanta, Georgia 30322, USA}
\email{ono@mathcs.emory.edu}

\address{School of Mathematics and Physics,
The University of Queensland, Brisbane, Australia}
\email{o.warnaar@maths.uq.edu.au}

\thanks{The first two authors are supported by the NSF, and the third 
author is supported by the Australian Research Council.
The second author also thanks the support of the Asa Griggs Candler Fund.}

\dedicatory{In memory of Basil Gordon and Alain Lascoux}


\begin{abstract} 
The two Rogers--Ramanujan $q$-series
\[
\sum_{n=0}^{\infty}\frac{q^{n(n+\sigma)}}{(1-q)\cdots (1-q^n)},
\]
where $\sigma=0,1$, play many roles in mathematics and physics.
By the Rogers--Ramanujan identities, they are essentially 
modular functions.
Their quotient, the Rogers--Ramanujan continued fraction, 
has the special property that its \textit{singular values}
are algebraic integral units. 
We find a framework which extends the Rogers--Ramanujan identities to
doubly-infinite families of $q$-series identities.
If $a\in \{1,2\}$ and  $m,n \geq 1$, then we have
\[
\sum_{\substack{\la \\[1pt] \la_1\leq m}}
q^{a\abs{\la}} P_{2\la}(1,q,q^2,\dots;q^{n})
=\textup{``infinite product modular function'',}
\]
where the $P_{\lambda}(x_1,x_2,\dots;q)$ are 
\textit{Hall--Littlewood polynomials}.
These $q$-series
are specialized characters of affine Kac--Moody
algebras. 
Generalizing the Rogers--Ramanujan continued fraction, 
we prove in the case of $\A_{2n}^{(2)}$ that
the relevant $q$-series quotients are
integral units.
\end{abstract}

\maketitle

\section{Introduction}\label{Intro}
The Rogers--Ramanujan (RR) identities \cite{Rogers94}
\begin{equation}\label{G}
G(q):=\sum_{n=0}^\infty\frac{q^{n^2}}{(1-q)\cdots(1-q^n)}
=\prod_{n=0}^\infty\frac{1}{(1-q^{5n+1})(1-q^{5n+4})}
\end{equation}
and 
\begin{equation}\label{H}
H(q):=\sum_{n=0}^\infty\frac{q^{n^2+n}}{(1-q)\cdots(1-q^n)}
=\prod_{n=0}^\infty\frac{1}{(1-q^{5n+2})(1-q^{5n+3})}
\end{equation}
play many roles in mathematics and physics.
They are essentially modular functions, and their ratio $H(q)/G(q)$ 
is the famous Rogers--Ramanujan $q$-continued fraction  
\begin{equation}\label{GH}
\frac{H(q)}{G(q)}=\cfrac{1}{1+\cfrac{q}{1+\cfrac{q^2}{1+\cfrac{q^3}
{\,\ddots}}}}\,.
\end{equation}

The \textit{golden ratio} $\phi$ satisfies $H(1)/G(1)=1/\phi=(-1+\sqrt{5})/2$.
Ramanujan computed further values
such as\footnote{He offered this value in his first letter to Hardy 
(see p.~29 of \cite{Ramanujan}).}
\begin{equation}\label{firstletter}
\eup^{-\tfrac{2\pi}{5}}\cdot \frac{H(\eup^{-2\pi})}{G(\eup^{-2\pi})}=
\sqrt{\frac{5+\sqrt{5}}{2}}-\frac{\sqrt{5}+1}{2}.
\end{equation}
The minimal polynomial  of this value is
\[
x^4+2x^3-6x^2-2x+1,
\]
which shows that it is an algebraic integral unit. 
All of Ramanujan's evaluations are such units.

Ramanujan's evaluations inspired early work by Watson 
\cite{MM,Watson1,Watson2} and Ramanathan \cite{Ramanathan}.
Then in 1996, Berndt, Chan and Zhang \cite{BerndtChan}\footnote{Cais 
and Conrad \cite{CaisConrad} and Duke \cite{Duke} later revisited 
these results from the perspective of arithmetic geometry and the symmetries 
of the regular icosahedron respectively.} 
finally obtained  general theorems concerning such values. 
The theory pertains to values at $q:=\eup^{2\pi\iup\tau}$, where the 
$\tau$ are quadratic irrational points in the upper-half of 
the complex plane.  We refer to such a point $\tau$ as a \textit{CM point} with 
\textit{discriminant} $-D<0$, where $-D$ is the
discriminant of the minimal polynomial of $\tau$.  The
corresponding evaluation is known as a \textit{singular value}.
Berndt, Chan and Zhang proved that the singular values
$q^{-1/60}G(q)$ and $q^{11/60}H(q)$ 
are algebraic numbers
in abelian extensions of  $\Rat(\tau)$ which satisfy the exceptional property  
(see \cite[Theorem 6.2]{BerndtChan}) that their ratio 
$q^{{1/5}}H(q)/G(q)$ is an algebraic integral unit which generates 
specific abelian extensions of $\Rat(\tau)$.

\smallskip
\begin{remark}
The individual values of $q^{-1/60}G(q)$ and 
$q^{11/60}H(q)$ generically are not algebraic integers. 
For example, in \eqref{firstletter} we have $\tau=\iup$, 
and the numerator and denominator 
\[
q^{-\tfrac{1}{60}}G(q)=
\sqrt[4]{\frac{1+3\sqrt{5}+2\sqrt{10+2\sqrt{5}}}{10}}
\quad\text{and}\quad 
q^{\tfrac{11}{60}}H(q)=
\sqrt[4]{\frac{1+3\sqrt{5}-2\sqrt{10+2\sqrt{5}}}{10}}
\]
share the  minimal polynomial $625x^{16}-250x^{12}-1025x^8-90x^4+1.$
\end{remark}

In addition to the algebraic properties described above, 
\eqref{G} and \eqref{H} have been related to a 
large number of different areas of mathematics.
They were were first recognized by MacMahon and Schur 
as identities for integer partitions \cite{MacMahon60,Schur17},
but have since been linked to algebraic geometry \cite{BMS13,GOR13}, 
$\mathrm{K}$-theory \cite{DS94}, conformal field 
theory \cite{BM98,KMM95,LP85}, 
group theory \cite{Fulman00},
Kac--Moody, Virasoro, vertex and double affine Hecke algebras 
\cite{CF12,FF93,LM78a,LM78b,LW81a,LW81b,LW82,LW84},
knot theory \cite{AD11,Hikami03,Hikami06}, 
modular forms \cite{BCFK, BFM, BF0, BF, BM, BO, BOR08, Ono09},
orthogonal polynomials \cite{AI83,Bressoud81,GIS99},
statistical mechanics \cite{ABF84,Baxter81}, 
probability \cite{Fulman01} and
transcendental number theory \cite{RS81}.

In 1974 Andrews \cite{Andrews74} extended  \eqref{G} and \eqref{H}
to an infinite family of Rogers--Rama\-nu\-jan-type identities by 
proving that
\begin{equation}\label{Eq_AG}
\sum_{r_1\geq\dots\geq r_m\geq 0}
\frac{q^{r_1^2+\cdots+r_m^2+r_i+\cdots+r_m}}
{(q)_{r_1-r_2}\cdots(q)_{r_{m-1}-r_m}(q)_{r_m}}=
\frac{(q^{2m+3};q^{2m+3})_{\infty}}{(q)_{\infty}}\, \cdot
\theta(q^i;q^{2m+3}),
\end{equation}
where $1\leq i\leq m+1$.
As usual, here we have that
\[
(a)_k=(a;q)_k:= \begin{cases} 
(1-a)(1-aq)\cdots (1-aq^{k-1}) \quad &\text{if $k\geq 0$},\\[2mm]
\displaystyle \prod_{j=0}^{\infty} (1-aq^j)  &\text{if $k=\infty$},
\end{cases}
\]
and 
\[
\theta(a;q):=(a;q)_{\infty}(q/a;q)_{\infty}
\]
is a modified theta function.
The identities \eqref{Eq_AG}, which can be viewed as the analytic counterpart 
of Gordon's partition theorem \cite{Gordon61}, are now commonly referred to
as the Andrews--Gordon (AG) identities.

\smallskip
\begin{remark}
The specializations of $\theta(a;q)$  in \eqref{Eq_AG}
are (up to powers of $q$) modular functions, where 
$q:=\eup^{2\pi\iup\tau}$ and $\tau$ is any complex point with $\im(\tau)>0$. 
It should be noted 
that this differs from our use of $q$ and $\tau$ above where we 
required $\tau$ to be a quadratic irrational point. 
Such infinite product modular functions
were studied extensively by  Klein and Siegel.
\end{remark}

There are numerous algebraic interpretations of the 
Rogers--Ramanujan and Andrews--Gordon identities.
For example, the above-cited papers by Milne, Lepowsky and Wilson 
show that they arise, up to a factor $(q;q^2)_{\infty}$,
as principally specialized characters of 
integrable highest-weight modules of the affine Kac--Moody algebra 
$\A_1^{(1)}$. Similarly, Feigin and Frenkel proved the
Rogers--Ramanujan and Andrews--Gordon identities by 
considering certain irreducible minimal representations of the 
Virasoro algebra \cite{FF93}.
We should also mention the much larger program by Lepowsky 
and others on combinatorial and algebraic extensions
of Rogers--Ramanujan-type identities, leading to the introduction of
$Z$-algebras for all affine Lie algebras,
vertex-operator-theoretic proofs of generalized Rogers--Ramanujan identities, 
and Rogers--Ramanujan-type identities for arbitrary affine Lie algebras
in which, typically, the sum side is replaced by a combinatorial sum,
see e.g., \cite{FLM88,Lepowsky82,MP96} and references therein.

In this paper we have a similar but distinct aim, namely to find a 
concrete framework of Rogers--Ramanujan type identities in the
$q$-series sense of ``infinite sum $=$ infinite product'',
where the infinite products  
arise as specialized characters of appropriately 
chosen affine Lie algebras $\mathrm{X}_N^{(r)}$ for arbitrary $N$.
Such a general framework would give new  connections between 
Lie algebras and the theory of modular functions.

In \cite{ASW99} (see also \cite{FFW08,Warnaar06}) some results 
concerning the above question were obtained, resulting in 
Rogers--Ramanujan-type identities for $\A_2^{(1)}$. 
The approach of \cite{ASW99} 
does not in any obvious manner extend to $\A_n^{(1)}$ for all $n$,
and this paper aims to give a more complete answer.
By using a level-$m$ Rogers--Selberg identity for the root system
$\C_n$ as recently obtain by Bartlett and the third author \cite{BW13},
we show that the Rogers--Ramanujan and Andrews--Gordon identities
are special cases 
of a doubly-infinite family of $q$-identities arising from the 
Kac--Moody algebra $\A_{2n}^{(2)}$ for arbitrary $n$.
In their most compact form, the ``sum-sides'' are expressed in terms of 
Hall--Littlewood polynomials $P_{\la}(x;q)$
evaluated at infinite geometric progressions (see Section~\ref{Sec_HL}
 for definitions and further details), and the 
``product-sides'' are essentially products of modular theta functions.
We shall present four pairs $(a,b)$ such that for all $m,n \geq 1$
we have an identity of the form
\[
\sum_{\substack{\la \\[1pt] \la_1\leq m}}
q^{a\abs{\la}} P_{2\la}(1,q,q^2,\dots;q^{2n+b}) \\
=\text{``infinite product modular function''}.
\]

To make this precise, we fix notation for \textit{integer partitions}, 
nonincreasing sequences of nonnegative integers with at most finitely many 
nonzero terms.
For a partition $\la=(\la_1,\la_2,\dots)$, we let 
$\abs{\la}:=\la_1+\la_2+\cdots$, and we let
$2\la:=(2\la_1,2\la_2,\dots)$. 
We also require $\la'$, the \textit{conjugate} of $\la$, the partition 
which is obtained by transposing the Ferrers--Young diagram of $\la$.
Finally, for convenience we let
\begin{equation}
\theta(a_1,\dots,a_k;q):=\theta(a_1;q)\cdots\theta(a_k;q).
\end{equation}

\begin{example}
If $\la=(5,3,3,1)$, then we have that $\abs{\la}=12$, $2\la=(10,6,6,2)$
and $\la'=(4,3,3,1,1)$.
\end{example}

Using this notation, we have the following pair of doubly-infinite 
Rogers--Ramanujan type identities which correspond to specialized 
characters of $\A_{2n}^{(2)}$.

\begin{theorem}[$\A_{2n}^{(2)}$ RR and AG identities]\label{Thm_Main}
If $m$ and $n$ are positive integers and $\kappa:=2m+2n+1$, then
we have that
\begin{subequations}\label{Eq_RR-A2n2}
\begin{align}\label{Eq_RR-A2n2a}
\sum_{\substack{\la \\[1pt] \la_1\leq m}}
q^{\abs{\la}} & P_{2\la}\big(1,q,q^2,\dots;q^{2n-1}\big) \\
&=\frac{(q^{\kappa};q^{\kappa})_{\infty}^n}{(q)_{\infty}^n}\cdot 
\prod_{i=1}^n  \theta\big(q^{i+m};q^{\kappa}\big)
\prod_{1\leq i<j\leq n} 
\theta\big(q^{j-i},q^{i+j-1};q^{\kappa}\big) \notag \\
&=\frac{(q^{\kappa};q^{\kappa})_{\infty}^m}{(q)_{\infty}^m}\cdot 
\prod_{i=1}^m  \theta\big(q^{i+1};q^{\kappa}\big)
\prod_{1\leq i<j\leq m} 
\theta\big(q^{j-i},q^{i+j+1};q^{\kappa}\big), \notag
\intertext{and}
\label{Eq_RR-A2n2b}
\sum_{\substack{\la \\[1pt] \la_1\leq m}}
q^{2\abs{\la}} & P_{2\la}\big(1,q,q^2,\dots;q^{2n-1}\big) \\
&=\frac{(q^{\kappa};q^{\kappa})_{\infty}^n}{(q)_{\infty}^n}\cdot 
\prod_{i=1}^n  \theta\big(q^i;q^{\kappa}\big)
\prod_{1\leq i<j\leq n} 
\theta\big(q^{j-i},q^{i+j};q^{\kappa}\big) \notag \\
&=\frac{(q^{\kappa};q^{\kappa})_{\infty}^m}{(q)_{\infty}^m}\cdot 
\prod_{i=1}^m  \theta\big(q^i;q^{\kappa}\big)
\prod_{1\leq i<j\leq m} 
\theta\big(q^{j-i},q^{i+j};q^{\kappa}\big). \notag
\end{align}
\end{subequations}
\end{theorem}

\smallskip
\begin{remarks}
(1) When $m=n=1$, Theorem~\ref{Thm_Main} gives the 
Rogers--Ramanujan identities \eqref{G} and \eqref{H}. 
The summation defining the series is over the empty partition,
$\la=0$, and partitions consisting of $n$ copies of 1, i.e.,
$\la=(1^n)$.
Since
\[
q^{(\sigma+1)|(1^n)|} P_{(2^n)}(1,q,q^2,\dots;q)=
\frac{q^{n(n+\sigma)}}{(1-q)\cdots (1-q^n)},
\]
identities \eqref{G} and \eqref{H} thus follow from 
Theorem~\ref{Thm_Main} by letting $\sigma=0,1$.

\noindent
(2) When $n=1$, Theorem~\ref{Thm_Main}
 gives the $i=1$ and the $i=m+1$ instances of the Andrews--Gordon identities 
in a representation due to Stembridge \cite{Stembridge90} 
(see also Fulman \cite{Fulman00}).
The equivalence with \eqref{Eq_AG}
follows from the specialization formula \cite[p. 213]{Macdonald95}
\[
q^{(\sigma+1)\abs{\la}} P_{2\la}(1,q,q^2,\dots;q)=
\prod_{i\geq 1} \frac{q^{r_i(r_i+\sigma)}}{(q)_{r_i-r_{i+1}}},
\]
where $r_i:=\la'_i$. 
Note that $\la_1\leq m$ implies that $\la'_i=r_i=0$ for $i>m$.

\noindent
(3) We note the beautiful level-rank duality exhibited by the products
on the right-hand sides of the expressions in Theorem~\ref{Thm_Main}
(especially those of \eqref{Eq_RR-A2n2b}). 

\noindent
(4) In the next section we shall show that the more general 
series
\begin{equation}\label{star}
\sum_{\substack{\la \\[1pt] \la_1\leq m}}
q^{(\sigma+1)\abs{\la}} P_{2\la}\big(1,q,q^2,\dots;q^{n}\big) 
\end{equation}
are also expressible in terms of $q$-shifted factorials, allowing for
a formulation of Theorem~\ref{Thm_Main} (see Lemma~\ref{TwoPointFour})
which is independent of Hall--Littlewood polynomials.
\end{remarks}

\medskip
\begin{example} Here we illustrate Theorem~\ref{Thm_Main} when
$m=n=2$. 
Then
\eqref{Eq_RR-A2n2a} is
\[
\sum_{\substack{\la \\[1pt] \la_1\leq 2}}
q^{\abs{\la}} P_{2\la}\big(1,q,q^2,\dots;q^{3}\big)=
\prod_{n=1}^{\infty}\frac{(1-q^{9n})}{(1-q^n)},
\]
giving another expression for the $q$-series studied by Dyson in his ``A walk through 
Ramanujan's Garden" \cite{Garden}:

\smallskip
\noindent
\textit{``The end of the war was not in sight. 
In the evenings of that winter I kept sane by wandering in Ramanujan's garden. 
\dots I found a lot of identities of the sort that Ramanujan 
would have enjoyed. My favorite one was this one:
\[
\sum_{n=0}^{\infty}x^{n^2+n}\cdot
\frac{(1+x+x^2)(1+x^2+x^4)\cdots(1+x^n+x^{2n})}{(1-x)(1-x^2)\cdots
(1-x^{2n+1})}=\prod_{n=1}^{\infty}\frac{(1-x^{9n})}{(1-x^n)}.
\]
In the cold dark evenings, while I was scribbling these beautiful identities amid the
death and destruction of 1944, I felt close to Ramanujan. He had been scribbling
even more beautiful identities amid the death and destruction of 1917.''}

\smallskip

The series in \eqref{Eq_RR-A2n2b} is
\[
\sum_{\substack{\la \\[1pt] \la_1\leq 2}}
q^{2\abs{\la}} P_{2\la}\big(1,q,q^2,\dots;q^{3}\big)=
\prod_{n=1}^{\infty}
\frac{(1-q^{9n})(1-q^{9n-1})(1-q^{9n-8})}{(1-q^n)(1-q^{9n-4})(1-q^{9n-5})}.
\]
\end{example}

\medskip

We also have an even modulus analog of Theorem~\ref{Thm_Main}.
Surprisingly, the $a=1$ and $a=2$ cases correspond to 
dual affine Lie algebras, namely $\C_n^{(1)}$ and $\D_{n+1}^{(2)}$.

\begin{theorem}[$\C_n^{(1)}$ RR and AG identities]\label{Thm_Main2}
If $m$ and $n$ are positive integers and $\kappa:=2m+2n+2$, then we have that
\begin{align}\label{Eq_RR-Cn}
\sum_{\substack{\la \\[1pt] \la_1\leq m}}
q^{\abs{\la}} & P_{2\la}\big(1,q,q^2,\dots;q^{2n}\big) \\
&=\frac{(q^2;q^2)_{\infty}(q^{\kappa/2};q^{\kappa/2})_{\infty}
(q^{\kappa};q^{\kappa})_{\infty}^{n-1}}{(q)_{\infty}^{n+1}}  
\cdot 
\prod_{i=1}^n  \theta\big(q^i;q^{\kappa/2}\big)
\prod_{1\leq i<j\leq n} \theta\big(q^{j-i},q^{i+j};q^{\kappa}\big)
\notag \\
&=\frac{(q^{\kappa};q^{\kappa})_{\infty}^m}{(q)_{\infty}^m} 
\cdot \prod_{i=1}^m  \theta\big(q^{i+1};q^{\kappa}\big)
\prod_{1\leq i<j\leq m} 
\theta\big(q^{j-i},q^{i+j+1};q^{\kappa}\big). \notag
\end{align}
\end{theorem}

\begin{theorem}[$\D_{n+1}^{(2)}$ RR and AG identities]\label{Thm_Main3}
If $m$ and $n$ are positive integers such that $n\geq 2$, and $\kappa:=2m+2n$, then we have that
\begin{align}\label{Eq_RR-Dn}
\sum_{\substack{\la \\[1pt] \la_1\leq m}}
q^{2\abs{\la}} & P_{2\la}\big(1,q,q^2,\dots;q^{2n-2}\big) \\
&=\frac{(q^{\kappa};q^{\kappa})_{\infty}^n}
{(q^2;q^2)_{\infty}(q)_{\infty}^{n-1}} 
\cdot \prod_{1\leq i<j\leq n} \theta\big(q^{j-i},q^{i+j-1};q^{\kappa}\big) \notag \\
&=\frac{(q^{\kappa};q^{\kappa})_{\infty}^m}{(q)_{\infty}^m} 
\cdot \prod_{i=1}^m  \theta\big(q^i;q^{\kappa}\big)
\prod_{1\leq i<j\leq m} 
\theta\big(q^{j-i},q^{i+j};q^{\kappa}\big). \notag
\end{align}
\end{theorem}

\smallskip
\begin{remarks}
(1) The $(m,n)=(1,2)$ case of \eqref{Eq_RR-Dn} is equivalent to 
Milne's modulus $6$ Rogers--Ramanujan identity \cite[Theorem 3.26]{Milne94}.

\noindent
(2) If we take $m=1$ in \eqref{Eq_RR-Cn} (with $n\mapsto n-1$)
and \eqref{Eq_RR-Dn}, and apply formula \eqref{Eq_Q2r} below (with $\delta=0$), 
we obtain the $i=1,2$ cases of Bressoud's even modulus identities
\cite{Bressoud80}
\begin{equation}\label{Eq_Bressoud}
\sum_{r_1\geq\dots\geq r_n\geq 0}
\frac{q^{r_1^2+\cdots+r_n^2+r_i+\cdots+r_n}}
{(q)_{r_1-r_2}\cdots(q)_{r_{n-1}-r_n}(q^2;q^2)_{r_n}} 
=\frac{(q^{2n+2};q^{2n+2})_{\infty}}{(q)_{\infty}}\, \cdot 
\theta(q^i;q^{2n+2}).
\end{equation}
\end{remarks}

\medskip

By combining \eqref{Eq_RR-A2n2}--\eqref{Eq_RR-Dn},
we obtain an identity of ``mixed'' type.
\begin{corollary}
If $m$ and $n$ are positive integers and $\kappa:=2m+n+2$, then for
$\sigma=0,1$  we have that
\begin{align}\label{Eq_mixed}
\sum_{\substack{\la \\[1pt] \la_1\leq m}}
q^{(\sigma+1)\abs{\la}} P_{2\la}\big(1,&q,q^2,\dots;q^n\big) \\
&=\frac{(q^{\kappa};q^{\kappa})_{\infty}^m}{(q)_{\infty}^m} 
\cdot \prod_{i=1}^m  \theta\big(q^{i-\sigma+1};q^{\kappa}\big)
\prod_{1\leq i<j\leq m} 
\theta\big(q^{j-i},q^{i+j-\sigma+1};q^{\kappa}\big).
\end{align}
\end{corollary}

Identities for $\A_{n-1}^{(1)}$ also exist, 
although their formulation is perhaps slightly less satisfactory.
We have the following ``limiting'' Rogers--Ramanujan type identities.

\begin{theorem}[$\A_{n-1}^{(1)}$ RR and AG identities]\label{Thm_Main4}
If $m$ and $n$ are positive integers and $\kappa:=m+n$, then we have that
\begin{align*}
\lim_{r\to\infty}
q^{-m\binom{r}{2}} 
P_{(m^r)}(1,q,q^2,\dots;q^n)
&=\frac{(q^{\kappa};q^{\kappa})_{\infty}^{n-1}}{(q)_{\infty}^n}
\cdot \prod_{1\leq i<j\leq n} \theta(q^{j-i};q^{\kappa}) \\
&=\frac{(q^{\kappa};q^{\kappa})_{\infty}^{m-1}}{(q)_{\infty}^m} 
\cdot \prod_{1\leq i<j\leq m} \theta(q^{j-i};q^{\kappa}).
\end{align*}
\end{theorem}

Now we turn to the question of whether the new $q$-series
appearing in these theorems, which arise from the 
Hall--Littlewood polynomials, enjoy the same algebraic properties as
\eqref{G}, \eqref{H}, and the Rogers--Ramanujan continued fraction. 
As it turns out they do: their singular values are algebraic numbers.
Moreover, we characterize those ratios which simplify 
to algebraic integral units.

To make this precise,
we recall that $q=\eup^{2\pi\iup\tau}$ for $\im(\tau)>0$, and that
$m$ and $n$ are arbitrary positive integers.
The auxiliary parameter $\kappa=\kappa_*(m,n)$ in 
Theorems~\ref{Thm_Main}, \ref{Thm_Main2} and \ref{Thm_Main3}  
is defined as follows:
\begin{equation}\label{Eq_kappa}
\kappa=
\begin{cases}
\kappa_1(m,n):=2m+2n+1 & \text{for $\A_{2n}^{(2)}$} \\
\kappa_2(m,n):=2m+2n+2 & \text{for $\C_n^{(1)}$} \\
\kappa_3(m,n):=2m+2n   & \text{for $\D_{n+1}^{(2)}$}.
\end{cases}
\end{equation}

\smallskip
\begin{remark} The parameter $\kappa$ has a representation theoretic 
interpretation arising from the corresponding affine Lie 
algebra $X_N^{(r)}$ (see Section~\ref{Sec_Pf}).
It turns out that
\[
\kappa_*(m,n)=\frac{2}{r}\big(\lev(\Lambda)+h^{\vee}\big),
\]
where $\lev(\Lambda)$ is the level of the corresponding representation, 
$h^{\vee}$ is the dual Coxeter number and $r$ is the tier number.
\end{remark}

To obtain algebraic values, we require  certain normalizations of these series.
The subscripts below correspond to the labelling in the theorems.
In particular,  $\Phi_{1a}$ and $\Phi_{1b}$ appear 
in Theorem~\ref{Thm_Main}, $\Phi_2$ is in Theorem~\ref{Thm_Main2},
and $\Phi_3$ is in Theorem~\ref{Thm_Main3}. 
Using this notation, the series are
\begin{subequations}\label{Phi123}
\begin{align}\label{Phi1a}
\Phi_{1a}(m,n;\tau)&:=q^{\tfrac{mn(4mn-4m+2n-3)}{12\kappa}}
\sum_{\la:\:\la_1\leq m}q^{\abs{\la}} P_{2\la}(1,q,q^2,\dots;q^{2n-1})\\
\label{Phi1b}
\Phi_{1b}(m,n;\tau)&:=q^{\tfrac{mn(4mn+2m+2n+3)}{12\kappa}}
\sum_{\la:\:\la_1\leq m}q^{2\abs{\la}} P_{2\la}(1,q,q^2,\dots;q^{2n-1})\\
\label{Phi2}
\Phi_2(m,n;\tau)&:=q^{\tfrac{m(2n+1)(2mn-m+n-1)}{12\kappa}}
\sum_{\la:\:\la_1\leq m}q^{\abs{\la}} P_{2\la}(1,q,q^2,\dots;q^{2n})\\
\label{Phi3}
\Phi_3(m,n;\tau)&:=q^{\tfrac{m(2n-1)(2mn+n+1)}{12\kappa}}
\sum_{\la:\:\la_1\leq m}q^{2\abs{\la}} P_{2\la}(1,q,q^2,\dots;q^{2n-2}).
\end{align}
\end{subequations}

\smallskip
\begin{remarks}
(1) We note that $\Phi_3(m,n;\tau)$ is not well defined when $n=1$.

\noindent
(2) We note that the $\kappa_*(m,n)$ are odd in the $\A_{2n}^{(2)}$ cases,
and are even for the $\C_n^{(1)}$ and $\D_{n+1}^{(2)}$ cases.
This dichotomy will be important when seeking pairs of $\Phi_*$ whose 
singular values have ratios that are algebraic integral units.
\end{remarks}
\medskip

Our first result concerns the algebraicity of these values and their
Galois theoretic properties. 
We show that these values are in specific abelian extensions
of imaginary quadratic fields (see \cite{Borel, Cox} for background 
on the explicit class field theory of imaginary quadratic fields).
For convenience, if $-D<0$ is a discriminant, then we define
\[
D_0:=\begin{cases} \frac{D}{4} \ \ \ \ \ &\text{if $D\equiv 0\pmod 4$},\\
\frac{-D-1}{4} &\text{if $-D\equiv 1\pmod{4}$}.
\end{cases}
\]

\begin{theorem}\label{thm} 
Assume the notation above, and let $\kappa:=\kappa_*(m,n)$.
If $\kappa \tau$ is a CM point with discriminant $-D<0$, then
the following are true:

\begin{enumerate}
\item 
The singular value $\Phi_*(m,n;\tau)$ is an algebraic number.

\item The multiset
\[
\Big\{\Phi_*(m,n,\tau_Q/\kappa)_{(\gamma\cdot \delta_Q(\tau))}^{12\kappa}: 
\: (\gamma,Q) \in W_{\kappa,\tau}\times \mathcal{Q}_D\Big\}
\]
\textup{(}see Section~\ref{SiegelFunctions} for definitions\textup{)} 
consists of multiple copies of a Galois orbit over $\Rat$.
\item  If $\kappa>10$, $\abs{-D}>\kappa^4/2$, and $\gcd(D_0,\kappa)=1$, then
the multiset in (2) is a Galois orbit over $\Rat$. 
\end{enumerate}
\end{theorem}

\smallskip
\begin{remarks}
(1) For each pair of positive integers $m$ and $n$, the inequality in
Theorem~\ref{thm} (3) holds for all but finitely many discriminants.

\noindent
(2) In Section~\ref{SiegelFunctions} we will show that the values $\Phi_*(m,n;\tau)^{12\kappa}$ are
 in a distinguished class field over the ring class field $\Rat(j(\kappa^2\tau))$,
 where $j(\tau)$
is the usual Klein $j$-function. 

\noindent
(3) The $\Phi_*$ singular values do not in general contain full sets of 
Galois conjugates. In particular, the singular values in the multiset in
Theorem~\ref{thm} (2) generally require $q$-series which are not among the 
four families $\Phi_*$.
For instance, only the $i=1$ and $i=m+1$ cases of the Andrews--Gordon 
identities arise from specializations of $\Phi_{1a}$ and $\Phi_{1b}$ 
respectively. 
However, the values associated to the other AG identities arise as Galois 
conjugates of these specializations. One then naturally wonders whether 
there are even further families of identities, perhaps those which
can be uncovered by the theory of complex multiplication.

\noindent
(4) Although Theorem~\ref{thm} (3) indicates that the multiset in (2) 
is generically a single orbit of Galois conjugates, it turns out that there
are indeed situations where the set is more than a single copy of such 
an orbit. Indeed, the two examples in
Section~\ref{Sec_Examples} will be such accidents.
\end{remarks}

\bigskip
We now address the question of singular values and algebraic integral units. 
Although the singular values of $q^{-1/60}G(q)$ and $q^{11/60}H(q)$ are not 
generally algebraic integers, their denominators can be determined exactly,
and their ratios always are algebraic integral units. 
The series $\Phi_*$ exhibit similar behavior.
The following theorem determines the integrality properties of the singular 
values. Moreover, it gives algebraic integral unit ratios in the case of 
the $\A_{2n}^{(2)}$ identities, generalizing the case of the Rogers--Ramanujan
continued fraction.

\begin{theorem}\label{thm2}
Assume the notation and hypotheses in Theorem~\ref{thm}.
Then the following are true:
\begin{enumerate}
\item The singular value $1/\Phi_*(m,n;\tau)$ is an  algebraic integer.
\item The singular value $\Phi_*(m,n;\tau)$ is a unit over $\Z[1/\kappa]$.
\item The ratio $\Phi_{1a}(m,n;\tau)/\Phi_{1b}(m,n;\tau)$ is an algebraic integral unit.
\end{enumerate}
\end{theorem}

\smallskip
\begin{remarks}
(1) We have that $\Phi_{1a}(1,1;\tau)=q^{-1/60}G(q)$ and 
$\Phi_{1b}(1,1;\tau)=q^{11/60}H(q)$.
Therefore, Theorem~\ref{thm2}~(3) implies the theorem of 
Berndt, Chan, and Zhang that the ratios of these singular values---the 
singular values of the  Rogers--Ramanujan continued fraction---are 
algebraic integral units.

\noindent
(2) It is natural to ask whether Theorem~\ref{thm2} (3) is a special property 
enjoyed only by the $\A_{2n}^{(2)}$ identities.
More precisely, are ratios of singular values of further pairs of 
$\Phi_*$ series algebraic integral units?  
By Theorem~\ref{thm2} (2),
it is natural to restrict attention to cases where 
the $\kappa_*(m,n)$ integers agree.
Indeed, in these cases the singular values are already
integral over the common ring $\Z[1/\kappa]$.
Due to the parity of the $\kappa_*(m,n)$, the only other cases to consider are pairs involving
$\Phi_2$ and $\Phi_3$.
In Section~\ref{Sec_Examples} we give an example illustrating 
that such ratios for
$\Phi_2$ and $\Phi_3$ are not generically algebraic integral units.
\end{remarks}

\medskip

\begin{example}
In Section~\ref{Sec_Examples} we shall consider the $q$-series
$\Phi_{1a}(2,2;\tau)$ and $\Phi_{1b}(2,2;\tau)$. For $\tau=\iup/3$, the first 100 coefficients of the $q$-series
respectively  give the numerical approximations
\begin{align*}
\Phi_{1a}(2,2;\iup/3)&=0.577350\dots\stackrel{?}{=}\frac{1}{\sqrt{3}}\\
\Phi_{1b}(2,2;\iup/3)&=0.125340\dots
\end{align*}
Here we have that $\kappa_{1}(2,2)=9$.
Indeed, these values are not algebraic integers. 
Respectively, they are roots of
\begin{gather*}
3x^2-1\\
19683x^{18}-80919x^{12}+39366x^9+11016x^6+486x^3-1.
\end{gather*}
However, Theorem~\ref{thm2}~(2) applies, and we find that 
$\sqrt{3}\Phi_{1a}(2,2;\iup/3)$ and $\sqrt{3}\Phi_{1b}(2,2;\iup/3)$ are units.
Respectively, they are roots of 
\begin{gather*}
x-1\\
x^{18}+6 x^{15}-93 x^{12}-304 x^9+420 x^6-102 x^3+1.
\end{gather*}
Lastly, Theorem~\ref{thm2}~(3) applies, and so their ratio
\[
\frac{\Phi_{1a}(2,2;\iup/3)}{\Phi_{1b}(2,2;\iup/3)}=4.60627\dots
\]
is a unit. Indeed, it is a root of 
\[
x^{18}-102x^{15}+420x^{12}-304x^9-93x^6+6x^3+1.
\]
\end{example}

\medskip
The remainder of this paper is organized as follows.
In Section~\ref{Sec_HL}  we recall some basic definitions and facts
from the theory of Hall--Littlewood polynomials. We use these facts 
to give a different combinatorial 
representation for the left-hand side of \eqref{Eq_mixed}
(see Lemma~\ref{TwoPointFour}).
Then, in Sections~\ref{Sec_Pf} and \ref{Sec_Pf2}, we prove
Theorems~\ref{Thm_Main}--\ref{Thm_Main3} and
Theorem~\ref{Thm_Main4}, respectively.
The proofs require Weyl denominator formulas, Macdonald identities,
and a lemma for $\C_n$ hypergeometric series from \cite{BW13}.
We also interpret each of the theorems
from the point of view of representation theory.
Namely, we explain how these identities correspond to specialized
characters of Kac--Moody algebras of affine type.

As noted above, the specializations of the $\theta(a;q)$ that arise in these
identities are essentially
modular functions of the type which have been studied extensively by Klein and
Siegel. This is the key fact which we employ to derive 
Theorems~\ref{thm} and \ref{thm2}.
In Section~\ref{SiegelFunctions} we recall the Galois 
theoretic properties of the singular values of Siegel functions as 
developed by Kubert and Lang, and
in Section~\ref{Proofs} we prove Theorems~\ref{thm} and \ref{thm2}.
In the last section we conclude with a detailed discussion
of examples of Theorems~\ref{thm} and \ref{thm2}.

\smallskip
\begin{Note}
One of the referees asked about Rogers--Ramanujan identities for affine 
Lie algebras other than those considered in this paper.
It is indeed possible to extend some of our results to also include 
$\B_n^{(1)}$ and $\A_{2n-1}^{(2)}$. However, the results
of \cite{BW13}---which are essential in the proofs of 
Theorems~\ref{Thm_Main}--\ref{Thm_Main3}---are not strong
enough to also deal with these two Kac--Moody algebras.
In \cite{RW15} Eric Rains and the third author present an alternative
method for expressing characters of affine Lie algebras in terms of
Hall--Littlewood polynomials to that of \cite{BW13}. Their method 
employs what are known as virtual Koornwinder integrals 
\cite{Rains05,RV07} instead of the $\C_n$ Bailey lemma used in \cite{BW13}.
This results in several further Rogers--Ramanujan identities, including
identities for $\B_n^{(1)}$ and $\A_{2n-1}^{(2)}$. 
At this stage it is not clear to us how to deal with $\D_n^{(1)}$ or
any of the exceptional affine Lie algebras.
\end{Note}

\section*{Acknowledgements}\noindent
 The authors thank Edward Frenkel, James Lepowsky, Dong Hwa Shin, and Drew Sills for their comments
on a preliminary version of this paper.

\section{The Hall--Littlewood polynomials}\label{Sec_HL}

Let $\la=(\la_1,\la_2,\dots)$ be an \textit{integer partition} 
\cite{Andrews76}, a nonincreasing sequence of nonnegative integers
$\la_1\geq\la_2\geq\dots$ with only finitely nonzero terms.
The positive $\la_i$ are called the \textit{parts}
of $\la$, and the number of parts, denoted $l(\la)$, is the
\textit{length} of $\la$. 
The \textit{size} $\abs{\la}$ of $\la$ is the sum of its parts.
The Ferrers--Young diagram of $\la$ consists of $l(\la)$ left-aligned
rows of squares such that the $i$th row contains $\la_i$ squares.
For example, the Ferrers--Young diagram of $\nu=(6,4,4,2)$ 
of length $4$ and size $16$ is
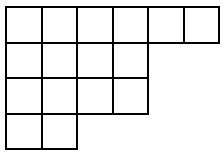
\[
\yng(6,4,4,2)
\]
The \textit{conjugate} partition $\la'$ corresponds to the transpose
of the Ferrers--Young diagram of $\la$.
For example, we have $\nu'=(4,4,3,3,1,1)$.
We define nonnegative integers $m_i=m_i(\la)$, for $i\geq 1$,
to be the \textit{multiplicities} of parts of size $i$, so that
$\abs{\la}=\sum_i i m_i$.
It is easy to see that $m_i=\la'_i-\la'_{i+1}$.
We say that a partition is \textit{even} if its parts are all even.
Note that $\la'$ is even if all multiplicities $m_i(\la)$
are even. The partition $\nu$ above is an even partition.
Given two partitions $\la,\mu$ we write $\mu\subseteq\la$ if the
diagram of $\mu$ is contained in the diagram of $\la$, or, equivalently,
if $\mu_i\leq \la_i$ for all $i$.
To conclude our discussion of partitions, we define the 
\textit{generalized $q$-shifted factorial}
\begin{equation}\label{Eq_blambda}
b_{\la}(q):=\prod_{i\geq 1} (q)_{m_i}=\prod_{i\geq 1} (q)_{\la'_i-\la'_{i+1}}.
\end{equation}
Hence, for $\nu$ as above we have $b_{\nu}(q)=(q)_1^2(q)_2$.

For a fixed positive integer $n$, let $x=(x_1,\dots,x_n)$.
Given a partition $\la$ such that $l(\la)\leq n$,
write $x^{\la}$ for the monomial $x_1^{\la_1}\dots x_n^{\la_n}$,
and define
\begin{equation}
v_{\la}(q)=\prod_{i=0}^n \frac{(q)_{m_i}}{(1-q)^{m_i}},
\end{equation}
where $m_0:=n-l(\la)$.
The \textit{Hall--Littlewood polynomial} $P_{\la}(x;q)$ 
is defined as the symmetric function \cite{Macdonald95}
\begin{equation}
P_{\la}(x;q)=\frac{1}{v_{\la}(q)}
\sum_{w\in\Symm_n} 
w\bigg(x^{\la}\prod_{i<j}\frac{x_i-qx_j}{x_i-x_j}\bigg),
\end{equation}
where the symmetric group $\Symm_n$ acts on $x$ by permuting the $x_i$.
It follows from the definition that $P_{\la}(x;q)$ is
a homogeneous polynomial of degree $\abs{\la}$, a fact used repeatedly
in the rest of this paper.
$P_{\la}(x;q)$ is defined to be identically $0$ if $l(\la)>n$.
The Hall--Littlewood polynomials may be extended
in the usual way to symmetric functions in
countably-many variables, see \cite{Macdonald95}.

Here we make this precise when $x$ is specialized to an infinite geometric progression.
For $x=(x_1,x_2,\dots)$ not necessarily finite, let $p_r$ be the $r$-th 
power sum symmetric function
\[
p_r(x)=x_1^r+x_2^r+\cdots,
\]
and $p_{\la}=\prod_{i\geq 1} p_{\la_i}$.
The power sums $\{p_{\la}(x_1,\dots,x_n)\}_{l(\la)\leq n}$ form a 
$\Rat$-basis of the ring of symmetric functions in $n$ variables.
If $\phi_q$ denotes the ring homomorphism $\phi_q(p_r)=p_r/(1-q^r)$,
then the \textit{modified Hall--Littlewood polynomials} $P'_{\la}(x;q)$ 
are defined as the image of the $P_{\la}(x;q)$ under $\phi_q$: 
\[
P'_{\la}=\phi_q\big(P_{\la}\big).
\]

We also require the Hall--Littlewood polynomials $Q_{\la}$ and $Q'_{\la}$ 
defined by
\begin{equation}\label{Eq_QQp}
Q_{\la}(x;q):=b_{\la}(q) P_{\la}(x;q)\quad\text{and}\quad
Q'_{\la}(x;q):=b_{\la}(q) P'_{\la}(x;q).
\end{equation}
Clearly, $Q'_{\la}=\phi_q\big(Q_{\la}\big)$.

Up to the point where the $x$-variables are specialized, our proof 
of Theorems~\ref{Thm_Main}--\ref{Thm_Main3}
will make use of the modified Hall--Littlewood polynomials, rather than the ordinary Hall--Littlewood polynomials.
Through specialization,  we arrive at $P_{\la}$ evaluated
at a geometric progression thanks to
\begin{equation}\label{Eq_PpP}
P_{\la}(1,q,q^2,\dots;q^n)=
P'_{\la}(1,q,\dots,q^{n-1};q^n),
\end{equation}
which readily follows from
\[
\phi_{q^n} \big( p_r(1,q,\dots,q^{n-1}) \big)
=\frac{1-q^{nr}}{1-q^r} \cdot \frac{1}{1-q^{nr}}
=p_r(1,q,q^2,\dots).
\]

From \cite{Kirillov00,WZ12} we may infer the following combinatorial
formula for the modified Hall--Littlewood polynomials:
\[
Q'_{\la}(x;q)=
\sum \prod_{i=1}^{\la_1}\prod_{a=1}^n
x_a^{\mu^{(a-1)}_i-\mu^{(a)}_i}
q^{\binom{\mu^{(a-1)}_i-\mu^{(a)}_i}{2}}
\qbin{\mu^{(a-1)}_i-\mu^{(a)}_{i+1}}{\mu^{(a-1)}_i-\mu^{(a)}_i}_q,
\]
where the sum is over partitions
$0=\mu^{(n)}\subseteq\cdots\subseteq\mu^{(1)}\subseteq\mu^{(0)}=\la'$
and
\[
\qbin{n}{m}_q=\begin{cases} \displaystyle \frac{(q)_{n}}{(q)_m(q)_{n-m}} &
\text{if $m\in\{0,1,\dots,n\}$} \\[3mm]
0 & \text{otherwise} \end{cases}
\]
is the usual $q$-binomial coefficient.
Therefore, by \eqref{Eq_blambda}--\eqref{Eq_PpP},
we have obtained the following combinatorial description of the
$q$-series we have assembled from the Hall--Littlewood polynomials.

\begin{lemma}\label{TwoPointFour}
If $m$ and $n$ are positive integers, then
\begin{multline}\label{Eq_comb}
\sum_{\substack{\la \\[1pt] \la_1\leq m}} 
q^{(\sigma+1)\abs{\la}} P_{2\la}(1,q,q^2,\dots;q^n)\\
=\sum \prod_{i=1}^{2m}\Bigg\{
\frac{q^{\frac{1}{2}(\sigma+1)\mu^{(0)}_i}}
{(q^n;q^n)_{\mu^{(0)}_i-\mu^{(0)}_{i+1}}}
\prod_{a=1}^n
q^{\mu_i^{(a)}+n\binom{\mu^{(a-1)}_i-\mu^{(a)}_i}{2}}
\qbin{\mu^{(a-1)}_i-\mu^{(a)}_{i+1}}{\mu^{(a-1)}_i-\mu^{(a)}_i}_{q^n}
\Bigg\},
\end{multline}
where the sum on the right is over partitions
$0=\mu^{(n)}\subseteq\cdots\subseteq\mu^{(1)}\subseteq\mu^{(0)}$
such that $(\mu^{(0)})'$ is even and $l(\mu^{(0)})\leq 2m$.
\end{lemma}

Lemma~\ref{TwoPointFour}
may be used to express the sum sides of 
\eqref{Eq_RR-A2n2}--\eqref{Eq_mixed} combinatorially.
Moreover, we have that
 \eqref{Eq_comb} generalizes the 
sums in (\ref{G}), (\ref{H}), and (\ref{Eq_AG}). To see this, we note that the above
simplifies for $n=1$ to
\[
\sum_{\substack{\la \\[1pt] \la_1\leq m}} 
q^{(\sigma+1)\abs{\la}} P_{2\la}(1,q,q^2,\dots;q)
=\sum
\prod_{i=1}^{2m} \frac{q^{\frac{1}{2}\mu_i(\mu_i+\sigma)}}{(q)_{\mu_i-\mu_{i+1}}}
\]
summed on the right over partitions $\mu$ of length at most $2m$ whose conjugates are even.
Such partitions are characterized by the restriction 
$\mu_{2i}=\mu_{2i-1}=:r_i$ so that we get
\[
\sum_{\substack{\la \\[1pt] \la_1\leq m}} 
q^{(\sigma+1)\abs{\la}} P_{2\la}(1,q,q^2,\dots;q)
=\sum_{r_1\geq\cdots\geq r_m\geq 0} \prod_{i=1}^m
\frac{q^{r_i(r_i+\sigma)}}{(q)_{r_i-r_{i+1}}}
\]
in accordance with \eqref{Eq_AG}.

If instead we consider $m=1$ and replace
$\mu^{(j)}$ by $(r_j,s_j)$ for $j\geq 0$,
we find
\begin{align*}
&\sum_{r=0}^{\infty} q^{(\sigma+1)r} P_{(2^r)}(1,q,q^2,\dots;q^n)\\
&\quad=\sum
\frac{q^{(\sigma+1)r_0}}{(q^n;q^n)_{r_0}}
\prod_{j=1}^n
q^{r_j+s_j+n\binom{r_{j-1}-r_j}{2}
+n\binom{s_{j-1}-s_j}{2}}
\qbin{r_{j-1}-s_j}{r_{j-1}-r_j}_{q^n} 
\qbin{s_{j-1}}{s_j}_{q^n} \\
&\quad=\frac{(q^{n+4};q^{n+4})_{\infty}}{(q)_{\infty}}\, 
\cdot \theta\big(q^{2-\sigma};q^{n+4}\big),
\end{align*}
where the second sum is over $r_0,s_0,\dots,r_{n-1},s_{n-1}$ such that
$r_0=s_0$, and $r_n=s_n:=0$.

We conclude this section with a remark about Theorem~\ref{Thm_Main4}.
Due to the occurrence of the limit, the left-hand side does not take the
form of the usual sum-side of a Rogers--Ramanujan-type identity.
For special cases it is, however, possible to eliminate the limit.
For example, for partitions of the form $(2^r)$ we found that
\begin{equation}\label{Eq_Q2r}
P_{(2^r)}(1,q,q^2,\dots;q^{2n+\delta}) 
=\sum_{r\geq r_1\geq\dots\geq r_n\geq 0}
\frac{q^{r^2-r+r_1^2+\cdots+r_n^2+r_1+\cdots+r_n}}
{(q)_{r-r_1}(q)_{r_1-r_2}
\cdots(q)_{r_{n-1}-r_n}(q^{2-\delta};q^{2-\delta})_{r_n}}
\end{equation}
for $\delta=0,1$.
This turns the $m=2$ case of Theorem~\ref{Thm_Main4} into
\[
\sum_{r_1\geq\dots\geq r_n\geq 0}
\frac{q^{r_1^2+\cdots+r_n^2+r_1+\cdots+r_n}}
{(q)_{r_1-r_2}\cdots(q)_{r_{n-1}-r_n}(q^{2-\delta};q^{2-\delta})_{r_n}} 
=\frac{(q^{2n+2+\delta};q^{2n+2+\delta})_{\infty}}
{(q)_{\infty}}\, \cdot \theta(q;q^{2n+2+\delta}).
\]
For $\delta=1$ this is the $i=1$ case of the
Andrews--Gordon identity \eqref{Eq_AG} (with $m$ replaced by $n$). 
For $\delta=0$ it corresponds to the $i=1$ case of 
\eqref{Eq_Bressoud}.
We do not know how to generalize \eqref{Eq_Q2r}
to arbitrary rectangular shapes.

\section{Proof of Theorems~\ref{Thm_Main}--\ref{Thm_Main3}}\label{Sec_Pf}
Here we prove Theorems~\ref{Thm_Main}--\ref{Thm_Main3}.
We begin by recalling key aspects of the classical works of Andrews 
and Watson which give hints of the generalizations we obtain.

\subsection{The Watson--Andrews approach}

In 1929 Watson proved the Rogers--Ramanujan identities \eqref{G} and \eqref{H}
by first proving a new basic hypergeometric series transformation between
a terminating balanced $_4\phi_3$ series and a terminating 
very-well-poised $_8\phi_7$ series \cite{Watson29}
\begin{multline}\label{Eq_Watson}
\frac{(aq,aq/bc)_N}{(aq/b,aq/c)_N}
\sum_{r=0}^N \frac{(b,c,aq/de,q^{-N})_r}
{(q,aq/d,aq/e,bcq^{-N}/a)_r} \, q^r \\
=\sum_{r=0}^N \frac{1-a q^{2r}}{1-a}\, \cdot
\frac{(a,b,c,d,e,q^{-N})_r}{(q,aq/b,aq/c,aq/d,aq/e)_r}
\cdot \bigg(\frac{a^2q^{N+2}}{bcde}\bigg)^r .
\end{multline} 
Here $a,b,c,d,e$ are indeterminates, $N$ is a nonnegative integer
and
\[
(a_1,\dots,a_m)_k:=(a_1,\dots,q_m;q)=(a_1;q)_k\cdots (a_m;q)_k.
\]
By letting $b,c,d,e$ tend to infinity and taking the
nonterminating limit $N\to\infty$,
Watson arrived at what is known as the Rogers--Selberg 
identity \cite{Rogers19,Selberg36}\footnote{Here and elsewhere in the paper
we ignore questions of convergence. From an analytic point of view, 
the transition from \eqref{Eq_Watson} to \eqref{Eq_RS} requires
the use of the dominated convergence theorem, imposing the 
restriction $\abs{q}<1$ on the Rogers--Selberg identity.
We however choose to view this identity as an identity between 
formal power series in $q$,
in line with the combinatorial and representation-theoretic 
interpretations of Rogers--Ramanujan-type identities.} 
\begin{equation}\label{Eq_RS}
\sum_{r=0}^{\infty} \frac{a^r q^{r^2}}{(q)_r} \\
=\frac{1}{(aq)_{\infty}} \sum_{r=0}^{\infty}
\frac{1-a q^{2r}}{1-a}\, \cdot\frac{(a)_r}{(q)_r}\,
\cdot (-1)^r a^{2r} q^{5\binom{r}{2}+2r}.
\end{equation}
For $a=1$ or $a=q$ the sum on the right can be expressed in product-form 
by the Jacobi triple-product identity
\[
\sum_{r=-\infty}^{\infty} (-1)^r x^r q^{\binom{r}{2}}
=(q)_{\infty}\, \cdot \theta(x;q),
\]
resulting in (\ref{G}) and (\ref{H}).

Almost 50 years after Watson's work, Andrews showed that the 
Andrews--Gordon identities \eqref{Eq_AG} for $i=1$ and $i=m+1$ follow in 
a similar way from a multiple series generalization of 
\eqref{Eq_Watson} in which the $_8\phi_7$ series on the right is replaced 
by a terminating very-well-poised $_{2m+6}\phi_{2m+5}$ series depending 
on $2m+2$ parameters instead of $b,c,d,e$ \cite{Andrews75}.
Again the key steps are to let all these parameters
tend to infinity, to take the nonterminating limit, 
and to then express the $a=1$ or $a=q$ instances of the resulting
sum as a product by the Jacobi triple-product identity.

Recently, Bartlett and the third author obtained an analog of
Andrews' multiple series transformation for the $\C_n$ root system 
\cite[Theorem 4.2]{BW13}. 
Apart from the variables $(x_1,\dots,x_n)$---which play the role of
$a$ in \eqref{Eq_Watson}, and are related to the underlying root 
system---the $\C_n$ Andrews transformation again contains $2m+2$
parameters.
Unfortunately, simply following the Andrews--Watson procedure
is no longer sufficient. In \cite{Milne94} Milne already 
obtained the $\C_n$ analogue of the Rogers--Selberg identity 
\eqref{Eq_RS} (the $m=1$ case of \eqref{Eq_Cn-RS-m} below) and 
considered specializations along the lines of Andrews and Watson. 
Only for $\C_2$ did this result in a Rogers--Ramanujan-type identity:
the modulus $6$ case of \eqref{Eq_RR-Dn} mentioned previously.

The first two steps towards a proof of \eqref{Eq_RR-A2n2}--\eqref{Eq_mixed},
however, are the same as those of Watson and Andrews:
we let all $2m+2$ parameters in 
the $\C_n$ Andrews transformation tend to infinity and take the 
nonterminating limit. Then, as shown in \cite{BW13}, the right-hand side 
can be expressed in terms of modified Hall--Littlewood polynomials, 
resulting in the level-$m$ $\C_n$ Rogers--Selberg identity
\begin{equation}\label{Eq_Cn-RS-m}
\sum_{\substack{\la \\[1pt] \la_1\leq m}}
q^{\abs{\la}} P'_{2\la}(x;q)=L_m^{(0)}(x;q),
\end{equation}
where
\[
L^{(0)}_m(x;q) :=
\sum_{r\in\NN^n}\frac{\Delta_{\C}(x q^r)}{\Delta_{\C}(x)}\,
\prod_{i=1}^n x_i^{2(m+1)r_i} q^{(m+1)r_i^2+n\binom{r_i}{2}} 
\cdot \prod_{i,j=1}^n \Big({-}\frac{x_i}{x_j}\Big)^{r_i} 
\frac{(x_i x_j)_{r_i}}{(q x_i/x_j)_{r_i}}.
\]
Here we have that
\[
\Delta_{\C}(x):=\prod_{i=1}^n (1-x_i^2)\prod_{1\leq i<j\leq n}
(x_i-x_j)(x_ix_j-1)
\]
is the $\C_n$ Vandermonde product, and $f(xq^r)$ is short\-hand for 
$f(x_1q^{r_1},\dots,x_nq^{r_n})$.

\smallskip
\begin{remark}
As mentioned previously, \eqref{Eq_Cn-RS-m} for $m=1$ is Milne's 
$\C_n$ Rogers--Selberg formula \cite[Corollary 2.21]{Milne94}.
\end{remark}

The strategy for the proofs of Theorems~\ref{Thm_Main}--\ref{Thm_Main3}
is now simple to describe.
By comparing the left-hand side of \eqref{Eq_Cn-RS-m} with that of 
\eqref{Eq_RR-A2n2}--\eqref{Eq_RR-Dn}, it follows that we should 
make the simultaneous substitutions
\begin{equation}\label{Eq_spec}
q\mapsto q^n,\qquad x_i\mapsto q^{(n+\sigma+1)/2-i}\;\; (1\leq i\leq n).
\end{equation}
Then, by the homogeneity and symmetry of the (modified) Hall--Littlewood 
polynomials and \eqref{Eq_PpP}, we have
\[
\sum_{\substack{\la \\[1pt] \la_1\leq m}}
q^{\abs{\la}} P'_{2\la}(x;q)  \longmapsto
\sum_{\substack{\la \\[1pt] \la_1\leq m}}
q^{(\sigma+1)\abs{\la}} P_{2\la}(1,q,q^2,\dots;q^n).
\]

Therefore, we wish to carry out these maneuvers and prove that the resulting
right-hand side can be described as a product of modified theta functions
in the four families in the theorems.
The problem we face is that making the substitutions \eqref{Eq_spec} 
in the right-hand side of \eqref{Eq_Cn-RS-m} 
and then writing the resulting $q$-series
in product form is very difficult.

To get around this problem, we take a rather different route and 
(up to a small constant) first double the rank of the 
underlying $\C_n$ root system
and then take a limit in which products of pairs of $x$-variables tend 
to one. To do so we require another result from \cite{BW13}.

First we  extend our earlier definition of the $q$-shifted
factorial to
\begin{equation}
(a)_k=(a)_{\infty}/(aq^k)_{\infty}.
\end{equation}
Importantly, we note that $1/(q)_k=0$ for $k$ a negative integer. Then,
for $x=(x_1,\dots,x_n)$, $p$ an integer such that $0\leq p\leq n$ and 
$r\in\Z^n$, we have
\begin{multline}\label{Eq_Def-Lp}
L^{(p)}_m(x;q):=
\sum_{r\in\Z^n} \frac{\Delta_{\C}(x q^r)}{\Delta_{\C}(x)}\,
\prod_{i=1}^n x_i^{2(m+p+1)r_i}
q^{(m+1)r_i^2+(n+p)\binom{r_i}{2}} \\
\times
\prod_{i=1}^n \prod_{j=p+1}^n
\Big({-}\frac{x_i}{x_j}\Big)^{r_i} 
\frac{(x_i x_j)_{r_i}}{(qx_i/x_j)_{r_i}}.
\end{multline}
Note that the summand of $L^{(p)}_m(x;q)$ vanishes 
if one of $r_{p+1},\dots,r_n<0$.

The following lemma will be crucial for our strategy to work.

\begin{lemma}[{\cite[Lemma~A.1]{BW13}}]
For $1\leq p\leq n-1$,
\begin{equation}\label{Eq_key}
\lim_{x_{p+1}\to x_p^{-1}} L^{(p-1)}_m(x;q)=
L^{(p)}_m(x_1,\dots,x_{p-1},x_{p+1},\dots,x_n;q).
\end{equation}
\end{lemma}

This will be the key to the proof of the generalized 
Rogers--Ramanujan identities of Theorems~\ref{Thm_Main}--\ref{Thm_Main3}
although the level of difficulty varies considerably from case to case.

We begin with the simplest proof, that of the $\C_n$ Rogers--Ramanujan
and Andrews--Gordon identities of Theorem~\ref{Thm_Main2}. 
Although this theorem may also be proved more directly by principally
specializing \cite[Theorem 1.1]{BW13} (more on this later), 
we take a more indirect approach in order to describe the general method 
using the simplest available example. 
For the $\A_{2n}^{(2)}$ and $\D_{n+1}^{(2)}$ Rogers--Ramanujan and 
Andrews--Gordon identities we have no analogues of \cite[Theorem 1.1]{BW13},
and in these cases we rely on to the method described below.

\subsection{Proof of Theorem~\ref{Thm_Main2}}
Here we carry out the strategy described in the previous section by 
making use of the $\C_n$ and $\B_n$ Weyl denominator formulas, 
and the $\D_{n+1}^{(2)}$ Macdonald identity.

\begin{proof}[Proof of Theorem~\ref{Thm_Main2}]

By iterating \eqref{Eq_key}, we have
\[
\lim_{y_1\to x_1^{-1}}\dots \lim_{y_n\to x_n^{-1}}
L^{(0)}_m(x_1,y_1,\dots,x_n,y_n)=L^{(n)}_m(x_1,\dots,x_n).
\]
Hence, after replacing $x\mapsto(x_1,y_1,\dots,x_n,y_n)$ in 
\eqref{Eq_Cn-RS-m} (which corresponds to the doubling of the rank
mentioned previously) and taking the $y_i\to x_i^{-1}$ limit 
for $1\leq i\leq n$, we find
\begin{multline}\label{Eq_CnmLa0}
\sum_{\substack{\la \\[1pt] \la_1\leq m}}
q^{\abs{\la}} P'_{2\la}(x^{\pm};q)=
\frac{1}{(q)_{\infty}^n\prod_{i=1}^n \theta(x_i^2;q)
\prod_{1\leq i<j\leq n} \theta(x_i/x_j,x_ix_j;q)} \\ \times
\sum_{r\in\Z^n} \Delta_{\C}(x q^r)
\prod_{i=1}^n x_i^{\kappa r_i-i+1} q^{\frac{1}{2}\kappa r_i^2-nr_i},
\end{multline}
where $\kappa=2m+2n+2$ and $f(x^{\pm})=f(x_1,x_1^{-1},\dots,x_n,x_n^{-1})$.
Next we make the simultaneous substitutions
\begin{equation}\label{Eq_subs}
q\mapsto q^{2n},\qquad x_i\mapsto q^{n-i+1/2}=:\hat{x}_i\;\; (1\leq i\leq n),
\end{equation}
which corresponds to \eqref{Eq_spec} with $(n,\sigma)\mapsto (2n,0)$.
By  the identity
\[
(q^{2n};q^{2n})_{\infty}^n\cdot \prod_{i=1}^n \theta(q^{2n-2i+1};q^{2n})
\cdot \prod_{1\leq i<j\leq n} \theta(q^{j-i},q^{2n-i-j+1};q^{2n})=
\frac{(q)_{\infty}^{n+1}}{(q^2;q^2)_{\infty}},
\]
and
\begin{align*}
&q^{2n\abs{\la}} P'_{2\la}(q^{n-1/2},q^{1/2-n},\dots,q^{1/2},q^{-1/2};q^{2n}) 
\hspace{-10mm}
&& \\
&\quad=q^{2n\abs{\la}} P'_{2\la}(q^{1/2-n},q^{3/2-n},\dots,q^{n-1/2};q^{2n}) 
\hspace{-10mm}
&& \text{by symmetry} \\
&\quad=q^{\abs{\la}} P'_{2\la}(1,q,\dots,q^{2n-1};q^{2n}) && \text{by homogeneity} \\ 
&\quad=q^{\abs{\la}} P_{2\la}(1,q,q^2,\dots;q^{2n}) && \text{by \eqref{Eq_PpP}},
\end{align*}
we obtain
\begin{equation}\label{Eq_Mdef}
\sum_{\substack{\la \\[1pt] \la_1\leq m}}
q^{\abs{\la}} P_{2\la}\big(1,q,q^2,\dots;q^{2n}\big) =
\frac{(q^2;q^2)_{\infty}}{(q)_{\infty}^{n+1}}\, \M,
\end{equation}
where 
\[
\M:=\sum_{r\in\Z^n} \Delta_{\C}(\hat{x} q^{2nr})\,
\prod_{i=1}^n \hat{x}_i^{\kappa r_i-i+1} q^{n\kappa r_i^2-2n^2 r_i}.
\]

We must express $\M$ in product form. 
As a first step, we use the $\C_n$ Weyl denominator
formula \cite[Lemma 2]{Krattenthaler99}
\begin{equation}\label{Eq_detC}
\Delta_{\C}(x)=\det_{1\leq i,j\leq n} \big(x_i^{j-1}-x_i^{2n-j+1}\big),
\end{equation}
as well as multilinearity, to write $\M$ as
\begin{equation}\label{Eq_begin}
\M=\det_{1\leq i,j\leq n} 
\bigg(\sum_{r\in\Z} \hat{x}_i^{\kappa r-i+1}
q^{n\kappa r^2-2n^2r}
\Big( (\hat{x}_i q^{2nr})^{j-1}-(\hat{x}_iq^{2nr})^{2n-j+1}\Big)\bigg).
\end{equation}
We now replace $(i,j)\mapsto (n-j+1,n-i+1)$ and, viewing the
resulting determinant as being of the form 
$\det\big(\sum_r u_{ij;r}-\sum_r v_{ij;r}\big)$, we change the summation index
$r\mapsto -r-1$ in the sum over $v_{ij;r}$. Then we find that
\begin{equation}\label{Eq_end}
\M=\det_{1\leq i,j\leq n}\bigg(
q^{a_{ij}} \sum_{r\in\Z} y_i^{2nr-i+1}
q^{2n\kappa\binom{r}{2}+\frac{1}{2}\kappa r} 
\Big((y_iq^{\kappa r})^{j-1}-(y_i q^{\kappa r})^{2n-j}\Big)\bigg),
\end{equation}
where $y_i=q^{\kappa/2-i}$ and 
$a_{ij}=j^2-i^2+(i-j)(\kappa+1)/2$.
Since the factor $q^{a_{ij}}$ does not contribute to the
determinant, we can apply the $\B_n$ Weyl denominator formula 
\cite{Krattenthaler99}
\begin{equation}\label{Eq_Bn-VdM}
\det_{1\leq i,j\leq n} \big(x_i^{j-1}-x_i^{2n-j}\big)=
\prod_{i=1}^n (1-x_i)
\prod_{1\leq i<j\leq n} (x_i-x_j)(x_ix_j-1)=:\Delta_{\B}(x)
\end{equation}
to obtain
\[
\M=\sum_{r\in\Z^n} \Delta_{\B}(y q^{\kappa r})
\prod_{i=1}^n y_i^{2nr_i-i+1} 
q^{2n\kappa\binom{r_i}{2}+\frac{1}{2}\kappa r_i}.
\]
By the $\D_{n+1}^{(2)}$ Macdonald identity \cite{Macdonald72}
\begin{multline*}
\sum_{r\in\Z^n} \Delta_{\B}(xq^r)
\prod_{i=1}^n x_i^{2nr_i-i+1} 
q^{2n\binom{r_i}{2}+\frac{1}{2}r_i} \\
=(q^{1/2};q^{1/2})_{\infty}(q)_{\infty}^{n-1}
\prod_{i=1}^n \theta(x_i;q^{1/2})_{\infty}
\prod_{1\leq i<j\leq n} \theta(x_i/x_j,x_ix_j;q)
\end{multline*}
with $(q,x)\mapsto (q^{\kappa},y)$ this yields
\begin{equation}\label{Eq_Mcn}
\M=(q^{\kappa/2};q^{\kappa/2})_{\infty}
(q^{\kappa};q^{\kappa})_{\infty}^{n-1}
 \prod_{i=1}^n  \theta\big(q^i;q^{\kappa/2}\big)
\prod_{1\leq i<j\leq n} \theta\big(q^{j-i},q^{i+j};q^{\kappa}\big),
\end{equation}
where we have also used the simple symmetry 
$\theta(q^{a-b};q^a)=\theta(q^b;q^a)$.
Substituting \eqref{Eq_Mcn} into \eqref{Eq_Mdef} proves the first equality 
of \eqref{Eq_RR-Cn}.

Establishing the second equality is a straightforward
exercise in manipulating infinite products, and we omit the details.
\end{proof}

\medskip

There is a somewhat different approach to \eqref{Eq_RR-Cn} based on
the representation theory of the affine Kac--Moody algebra 
$\C_n^{(1)}$ \cite{Kac90}. 
Let $I=\{0,1,\dots,n\}$, and
$\alpha_i$, $\alpha^{\vee}_i$ and $\La_i$ for $i\in I$
the simple roots, simple coroots and fundamental weights
of $\C_n^{(1)}$.
Let $\ip{\cdot}{\cdot}$ denote the usual pairing between the Cartan
subalgebra $\hfrak$ and its dual $\hfrak^{\ast}$, 
so that $\ip{\La_i}{\alpha_j^{\vee}}=\delta_{ij}$.
Finally, let $V(\La)$ be the integrable highest-weight module of 
$\C_n^{(1)}$ of highest weight $\La$ with character $\ch V(\La)$.  

The homomorphism
\begin{equation}\label{Eq_PS}
F_{\mathds{1}}:~\mathbb{C}[[\eup^{-\alpha_0},\dots,\eup^{-\alpha_n}]]
\to \mathbb{C}[[q]],\qquad 
F_{\mathds{1}}(\eup^{-\alpha_i})=q\quad\text{for all $i\in I$}
\end{equation}
is known as principal specialization \cite{Lepowsky79}.
Subject to this specialization, $\ch V(\La)$
admits a simple product form as follows.
Let $\rho$ be the Weyl vector (that is $\ip{\rho}{\alpha_i^{\vee}}=1$
for $i\in I$) and $\mult(\alpha)$ the multiplicity of $\alpha$.
Then \cite{Kac78, Lepowsky82} we have
\begin{equation}\label{Eq_Kac}
F_{\mathbbm{1}}\big(\eup^{-\La} \ch V(\La)\big)=\prod_{\alpha\in\Delta_+^{\vee}}
\bigg(\frac{1-q^{\ip{\La+\rho}{\alpha}}}
{1-q^{\ip{\rho}{\alpha}}}\bigg)^{\mult(\alpha)},
\end{equation}
where $\Delta_+^{\vee}$ is the set of positive coroots.
This result, which is valid for all types $\mathrm{X}_N^{(r)}$, 
can be rewritten in terms of theta functions.
Assuming $\C_n^{(1)}$ and setting
\begin{equation}\label{Eq_Lambda-parametrization}
\La=(\la_0-\la_1)\La_0+(\la_1-\la_2)\La_1+\cdots+(\la_{n-1}-\la_n)\La_{n-1}+\la_n\La_n,
\end{equation}
for $\la=(\la_0,\la_1,\dots,\la_n)$ a partition, this rewriting takes the form
\begin{multline}\label{Eq_Kac2}
F_{\mathbbm{1}}\big(
\eup^{-\La} \ch V(\La)\big)
=\frac{(q^2;q^2)_{\infty}(q^{\kappa/2};q^{\kappa/2})_{\infty}
(q^{\kappa};q^{\kappa})_{\infty}^{n-1}} 
{(q;q)_{\infty}^{n+1}}  \\ \times
\prod_{i=1}^n  \theta\big(q^{\la_i+n-i+1};q^{\kappa/2}\big) 
\prod_{1\leq i<j\leq n} 
\theta\big(q^{\la_i-\la_j-i+j},q^{\la_i+\la_j+2n+2-i-j};q^{\kappa}\big),
\end{multline}
where $\kappa=2n+2\la_0+2$.

The earlier product form now arises by recognizing (see e.g.,
\cite[Lemma 2.1]{BW13}) the right-hand side of \eqref{Eq_CnmLa0} as
\begin{equation}\label{Eq_char}
\eup^{-m\La_0} \ch V(m\La_0)
\end{equation}
upon the identification
\[
q=\eup^{-\alpha_0-2\alpha_1-\cdots-2\alpha_{n-1}-\alpha_n}
\quad\text{and}\quad
x_i=\eup^{-\alpha_i-\cdots-\alpha_{n-1}-\alpha_n/2}\;\;
(1\leq i\leq n).
\]
Indeed, the equality between the left-hand side of \eqref{Eq_CnmLa0}
and \eqref{Eq_char} is exactly the first part of the
previously mentioned \cite[Theorem 1.1]{BW13}.
Since \eqref{Eq_subs} corresponds exactly
to the principal specialization \eqref{Eq_PS},
it follows from \eqref{Eq_Kac2} with $\la=(m,0^n)$, that
\begin{multline*}
F_{\mathbbm{1}}\big(\eup^{-m\La_0} \ch V(m\La_0)\big) 
=
\frac{(q^2;q^2)_{\infty}(q^{\kappa/2};q^{\kappa/2})_{\infty}
(q^{\kappa};q^{\kappa})_{\infty}^{n-1}} 
{(q;q)_{\infty}^{n+1}}   \\ \times
\prod_{i=1}^n  \theta\big(q^{n-i+1};q^{\kappa/2}\big) 
\prod_{1\leq i<j\leq n} 
\theta\big(q^{j-i},q^{i+j};q^{\kappa}\big).
\end{multline*}
This representation-theoretic approach is not essentially different from 
our earlier $q$-series proof. The principal specialization formula 
\eqref{Eq_Kac2} itself is an immediate consequence of the $\D^{(2)}_{n+1}$ 
Macdonald identity, and if instead of the right-hand side of
\eqref{Eq_CnmLa0} we consider the more general
\begin{multline*}
\eup^{-\La} \ch V(\La)=
\frac{1}{(q)_{\infty}^n\prod_{i=1}^n \theta(x_i^2;q)
\prod_{1\leq i<j\leq n} \theta(x_i/x_j,x_ix_j;q)} \\ \times
\sum_{r\in\Z^n} 
\det_{1\leq i,j\leq n}
\Big( (x_i q^{r_i})^{j-\la_j-1}-(x_i q^{r_i})^{2n-j+\la_j+1}\Big)
\prod_{i=1}^n x_i^{\kappa r_i+\la_i-i+1} q^{\frac{1}{2}\kappa r_i^2-nr_i}
\end{multline*}
for $\kappa=2n+2\la_0+2$, then all of the steps carried out between 
\eqref{Eq_CnmLa0} and \eqref{Eq_Mcn} carry over to this more
general setting. The only notable changes are that
\eqref{Eq_begin} generalizes to
\[
\M=\det_{1\leq i,j\leq n} 
\bigg(\sum_{r\in\Z} \hat{x}_i^{\kappa r+\la_i-i+1}
q^{n\kappa r^2-2n^2r} \cdot
\Big( (\hat{x}_i q^{2nr})^{j-\la_j-1}-(\hat{x}_iq^{2nr})^{2n-j+\la_j+1}\Big)
\bigg),
\]
and that in \eqref{Eq_end} we have to redefine
$y_i$ as $q^{\kappa/2-\la_{n-i+1}-i}$, and  $a_{ij}$ as
$$j^2-i^2+(i-j)(\kappa+1)/2+(j-1/2)\la_{n-j+1}-(i-1/2)\la_{n-i+1}.$$

\subsection{Proof of Theorem~\ref{Thm_Main} \eqref{Eq_RR-A2n2a}}

Here we prove \eqref{Eq_RR-A2n2a} by making use of the $\B_n^{(1)}$
Macdonald identity.

\begin{proof}[Proof of Theorem~\ref{Thm_Main}\eqref{Eq_RR-A2n2a}]
Again we iterate \eqref{Eq_key}, but this time the variable $x_n$, remains
unpaired: 
\[
\lim_{y_1\to x_1^{-1}}\dots\lim_{y_{n-1}\to x_{n-1}^{-1}}
L^{(0)}_m(x_1,y_1,\dots,x_{n-1},y_{n-1},x_n)=L^{(n-1)}_m(x_1,\dots,x_n).
\]
Therefore, if we replace $x\mapsto(x_1,y_1,\dots,x_{n-1},y_{n-1},x_n)$ 
in \eqref{Eq_Cn-RS-m} (changing the rank from $n$ to $2n-1$) 
and take the $y_i\to x_i^{-1}$ limit for $1\leq i\leq n-1$, we obtain
\begin{align}\label{Eq_interm}
\sum_{\substack{\la \\[1pt] \la_1\leq m}} &
q^{\abs{\la}} P'_{2\la}\big(x_1^{\pm},\dots,x_{n-1}^{\pm},x_n;q\big) \\[-1mm]
&=\frac{1}{(q)_{\infty}^{n-1}(qx_n^2)_{\infty} 
\prod_{i=1}^{n-1}  (qx_i^{\pm}x_n,qx_i^{\pm 2})_{\infty}
\prod_{1\leq i<j\leq n-1} (qx_i^{\pm} x_j^{\pm})_{\infty}} \notag \\[1mm] 
&\qquad \times
\sum_{r\in\Z^n} \frac{\Delta_{\C}(x q^r)}{\Delta_{\C}(x)}\,
\prod_{i=1}^n \bigg({-}\frac{x_i^{\kappa}}{x_n}\bigg)^{r_i} 
q^{\frac{1}{2}\kappa r_i^2-\frac{1}{2}(2n-1)r_i} 
\frac{(x_i x_n)_{r_i}}{(qx_i/x_n)_{r_i}}, \notag
\end{align}
where $\kappa=2m+2n+1$, 
$(ax_i^{\pm})_{\infty}:=(ax_i)_{\infty}(ax_i^{-1})_{\infty}$ and
\[
(ax_i^{\pm}x_j^{\pm})_{\infty}:=
(ax_ix_j)_{\infty}(ax_i^{-1}x_j)_{\infty}
(ax_ix_j^{-1})_{\infty}(ax_i^{-1}x_j^{-1})_{\infty}.
\]
Recalling the comment immediately after \eqref{Eq_Def-Lp},
the summand of \eqref{Eq_interm} vanishes unless $r_n\geq 0$.

Let $\hat{x}:=(-x_1,\dots,-x_{n-1},-1)$ and
\begin{equation}\label{Eq_phi}
\phi_r=\begin{cases} 1 & \text{if $r=0$} \\ 2 & \text{if $r=1,2,\dots$.}
\end{cases}
\end{equation}
Letting $x_n$ tend to $1$ in \eqref{Eq_interm}, and using
\[
\lim_{x_n\to 1} 
\frac{\Delta_{\C}(x q^r)}{\Delta_{\C}(x)}\,
\prod_{i=1}^n \frac{(x_ix_n)_{r_i}}{(qx_i/x_n)_{r_i}}
=\phi_{r_n} 
\frac{\Delta_{\B}(\hat{x} q^r)}{\Delta_{\B}(\hat{x})},
\]
we find that
\begin{align*}
\sum_{\substack{\la \\[1pt] \la_1\leq m}} &
q^{\abs{\la}} P'_{2\la}\big(x_1^{\pm},\dots,x_{n-1}^{\pm},1;q\big) \\[-1mm] 
&=\frac{1}{(q)_{\infty}^n 
\prod_{i=1}^{n-1}  (qx_i^{\pm},qx_i^{\pm 2})_{\infty}
\prod_{1\leq i<j\leq n-1} (qx_i^{\pm} x_j^{\pm})_{\infty}} \\[1mm]
\qquad &\ \ \ \ \ \ \ \times
\sum_{r_1,\dots,r_{n-1}=-\infty}^{\infty} \sum_{r_n=0}^{\infty}
\phi_{r_n} \frac{\Delta_{\B}(\hat{x} q^r)}{\Delta_{\B}(\hat{x})}\,
\prod_{i=1}^n \hat{x}_i^{\kappa r_i} 
q^{\frac{1}{2}\kappa r_i^2-\frac{1}{2}(2n-1)r_i}.
\end{align*}
It is easily checked that the summand on the right (without the factor
$\phi_{r_n}$) is invariant under the variable change
$r_n\mapsto -r_n$. Using the elementary relations
\begin{equation}\label{Eq_simp}
\theta(-1;q)=2(-q)_{\infty}^2,\quad
(-q)_{\infty}(q;q^2)_{\infty}=1,\quad
\theta(z,-z;q)\theta(qz^2;q^2)=\theta(z^2),
\end{equation}
we can then simplify the above to obtain
\begin{align}\label{Eq_a2n2}
\sum_{\substack{\la \\[1pt] \la_1\leq m}} &
q^{\abs{\la}} P'_{2\la}\big(x_1^{\pm},\dots,x_{n-1}^{\pm},1;q\big) \\[-1mm] 
&=\frac{1}{(q)_{\infty}^n 
\prod_{i=1}^n \theta(\hat{x}_i;q)\theta(q\hat{x}_i^2;q^2)
\prod_{1\leq i<j\leq n} 
\theta(\hat{x}_i/\hat{x}_j,\hat{x}_i\hat{x}_j;q)} \notag \\[1mm] 
& \qquad \ \ \ \ \ \times
\sum_{r\in\Z^n} \Delta_{\B}(\hat{x} q^r)\,
\prod_{i=1}^n \hat{x}_i^{\kappa r_i-i+1} 
q^{\frac{1}{2}\kappa r_i^2-\frac{1}{2}(2n-1)r_i}. \notag 
\end{align}

The remainder of the proof is similar to that of \eqref{Eq_RR-Cn}.
We make the simultaneous substitutions
\begin{equation}\label{Eq_A2n2-spec}
q\mapsto q^{2n-1},\qquad x_i\mapsto q^{n-i}\;\; (1\leq i\leq n),
\end{equation}
so that from here on $\hat{x}_i:=-q^{n-i}$. 
By the identity
\begin{multline*}
(q^{2n-1};q^{2n-1})_{\infty}^n\prod_{i=1}^n 
\theta(-q^{n-i};q^{2n-1})\theta(q^{2n-2i+1};q^{4n-2}) \\ \times
\prod_{1\leq i<j\leq n} \theta(q^{j-i},q^{2n-i-j};q^{2n-1})=
2(q)_{\infty}^n
\end{multline*}
and \eqref{Eq_PpP}, we find that  
\[
\sum_{\substack{\la \\[1pt] \la_1\leq m}}
q^{\abs{\la}} P_{2\la}\big(1,q,q^2,\dots;q^{2n-1}\big) =
\frac{\M}{2(q)_{\infty}^n}, 
\] 
where we have that
\[
\M:=\sum_{r\in\Z^n} \Delta_{\B}\big(\hat{x} q^{(2n-1)r}\big)\,
\prod_{i=1}^n \hat{x}_i^{\kappa r_i-i+1}
q^{\frac{1}{2}(2n-1)\kappa r_i^2-\frac{1}{2}(2n-1)^2r_i}.
\]
By \eqref{Eq_Bn-VdM} and multilinearity, $\M$ can be rewritten in the form
\[
\M=\det_{1\leq i,j\leq n} 
\bigg(\sum_{r\in\Z} \hat{x}_i^{\kappa r-i+1}
q^{\frac{1}{2}(2n-1)\kappa r^2-\frac{1}{2}(2n-1)^2r} \cdot
\Big( \big(\hat{x}_i q^{(2n-1)r}\big)^{j-1}-
\big(\hat{x}_iq^{(2n-1)r}\big)^{2n-j}\Big)\bigg).
\]
Following the same steps that led from \eqref{Eq_begin} to \eqref{Eq_end},
we obtain
\begin{multline}\label{Eq_same}
\M=\det_{1\leq i,j\leq n}\bigg(
(-1)^{i-j} q^{b_{ij}} \sum_{r\in\Z} (-1)^r 
y_i^{(2n-1)r-i+1} q^{(2n-1)\kappa\binom{r}{2}}  \\ \times
\Big((y_iq^{\kappa r})^{j-1}-(y_i q^{\kappa r})^{2n-j}\Big)\bigg),
\end{multline}
where 
\begin{equation}\label{Eq_yb}
y_i=q^{\frac{1}{2}(\kappa+1)-i} \quad\text{and}\quad 
b_{ij}:=j^2-i^2+\frac{1}{2}(i-j)(\kappa+3).
\end{equation} 
Again, the factor $(-1)^{i-j} q^{b_{ij}}$ does not contribute, and so
\eqref{Eq_Bn-VdM} then gives
\[
\M=\sum_{r\in\Z^n} 
\Delta_{\B}(y_iq^{\kappa r})
\prod_{i=1}^n (-1)^{r_i}  
y_i^{(2n-1)r_i-i+1} q^{(2n-1)\kappa\binom{r_i}{2}}.  
\]
To complete the proof, we apply the following variant of the 
$\B_n^{(1)}$ Macdonald identity\footnote{The actual
$\B_n^{(1)}$ Macdonald identity has the restriction $\abs{r}\equiv 0\pmod{2}$
in the sum over $r\in\Z^n$, which eliminates the factor $2$
on the right. To prove the form used here it suffices to take the
$a_1,\dots,a_{2n-1}\to 0$ and $a_{2n}\to -1$ limit in
Gustafson's multiple $_6\psi_6$ summation for the affine root system 
$\A_{2n-1}^{(2)}$, see \cite{Gustafson89}.}
\begin{multline}\label{Eq_Gustafson}
\sum_{r\in\Z^n}
\Delta_{\B}(xq^r) \prod_{i=1}^n (-1)^{r_i} 
x_i^{(2n-1)r_i-i+1} q^{(2n-1)\binom{r_i}{2}} \\
=2(q)_{\infty}^n 
\prod_{i=1}^n \theta(x_i;q)
\prod_{1\leq i<j\leq n} \theta(x_i/x_j,x_ix_j;q),
\end{multline}
with $(q,x)\mapsto (q^{\kappa},y)$.
\end{proof}

\medskip

Identity \eqref{Eq_RR-A2n2a} can be understood representation-theoretically, but
this time the relevant Kac--Moody algebra is $\A_{2n}^{(2)}$.
According to \cite[Lemma 2.3]{BW13}
the right-hand side of \eqref{Eq_a2n2}, with $\hat{x}$ interpreted 
(not
as $\hat{x}=(-x_1,\dots,-x_{n-1},-1)$) as
\[
\hat{x}_i=\eup^{-\alpha_0-\cdots-\alpha_{n-i}}\;\;(1\leq i\leq n)
\]
and $q$ as
\begin{equation}\label{Eq_null}
q=\eup^{-2\alpha_0-\cdots-2\alpha_{n-1}-\alpha_n},
\end{equation}
is the $\A_{2n}^{(2)}$ character 
\[
\eup^{-m\La_n} \ch V(m\La_n).
\]
The substitution \eqref{Eq_A2n2-spec} corresponds to 
\begin{equation}\label{Eq_specA}
\eup^{-\alpha_0}\mapsto -1 
\quad\text{and}\quad
\eup^{-\alpha_i}\mapsto q\;\;(1\leq i\leq n).
\end{equation}
Denoting this by $F$, it is not hard to derive the general specialization 
formula
\begin{equation}\label{Eq_not-principal}
F\big(\eup^{-\La} \ch V(\La)\big) 
=\frac{(q^{\kappa};q^{\kappa})_{\infty}^n}{(q)_{\infty}^n} 
\prod_{i=1}^n  \theta\big(q^{\la_i+n-i+1};q^{\kappa}\big) 
\prod_{1\leq i<j\leq n} 
\theta\big(q^{\la_i-\la_j-i+j},q^{\la_i+\la_j-i-j+2n+2};q^{\kappa}\big),
\end{equation}
where $\La$ is again parametrized as in 
\eqref{Eq_Lambda-parametrization}, $\la_0-\la_1$ is even\footnote{For 
$\la_0-\la_1$ odd, $F\big(\eup^{-\La} \ch V(\La)\big)=0$.}, 
and $\kappa=2n+\la_0+\la_1+1$.
For $\la=(m^{n+1})$ (so that $\La=m\La_n$) this yields the right-hand side
of \eqref{Eq_RR-A2n2a}.

\subsection{Proof of 
Theorem~\ref{Thm_Main} \eqref{Eq_RR-A2n2b}}

Here we prove the companion result to \eqref{Eq_RR-A2n2a}.

\begin{proof}[Proof of Theorem~\ref{Thm_Main} \eqref{Eq_RR-A2n2b}]
In \eqref{Eq_interm} we set $x_n=q^{1/2}$ so that
\begin{align*}
\sum_{\substack{\la \\[1pt] \la_1\leq m}} &
q^{\abs{\la}} P'_{2\la}\big(x_1^{\pm},\dots,x_{n-1}^{\pm},q^{1/2};q\big) \\[-1mm]
&=\frac{1}{(q)_{\infty}^{n-1}(q^2)_{\infty} 
\prod_{i=1}^{n-1}  (q^{3/2}x_i^{\pm},qx_i^{\pm 2})_{\infty}
\prod_{1\leq i<j\leq n-1} (qx_i^{\pm} x_j^{\pm})_{\infty}} \\[1mm] 
&\qquad \times
\sum_{r_1,\dots,r_{n-1}=-\infty}^{\infty} \sum_{r_n=0}^{\infty}
\frac{\Delta_{\C}(\hat{x} q^r)}{\Delta_{\C}(\hat{x})}\,
\prod_{i=1}^n (-1)^{r_i} \hat{x}_i^{\kappa r_i}
q^{\frac{1}{2}\kappa r_i^2-nr_i},
\end{align*}
where $\kappa=2m+2n+1$ and $\hat{x}=(x_1,\dots,x_{n-1},q^{1/2})$.
The $r_n$-dependent part of the summand is
\[
(-1)^{r_n} q^{\kappa\binom{r_n+1}{2}-nr_n}
\frac{1-q^{2r_n+1}}{1-q}\prod_{i=1}^{n-1}
\frac{x_iq^{r_i}-q^{r_n+1/2}}{x_i-q^{1/2}}\cdot
\frac{x_iq^{r_n+r_i+1/2}-1}{x_iq^{1/2}-1},
\]
which is readily checked to be invariant under the substitution
$r_n\mapsto -r_n-1$.
Hence
\begin{align*}
\sum_{\substack{\la \\[1pt] \la_1\leq m}} & 
q^{\abs{\la}} P'_{2\la}\big(x_1^{\pm},\dots,x_{n-1}^{\pm},q^{1/2};q\big) \\[-1mm]
&=\frac{1}{2(q)_{\infty}^n
\prod_{i=1}^{n-1} (-1) \theta(q^{1/2}x_i,x_i^2;q) 
\prod_{1\leq i<j\leq n-1} \theta(x_i/x_j,x_ix_j;q)} \\[1mm] 
&\qquad \times
\sum_{r\in\Z^n} 
\Delta_{\C}(\hat{x} q^r)
\prod_{i=1}^n (-1)^{r_i} \hat{x}_i^{\kappa r_i-i}
q^{\frac{1}{2}\kappa r_i^2-nr_i+\frac{1}{2}}.
\end{align*}
Our next step is to replace $x_i\mapsto x_{n-i+1}$ and
$r_i\mapsto r_{n-i+1}$. By $\theta(x;q)=-x\theta(x^{-1};q)$ and
\eqref{Eq_simp}, this leads to
\begin{align}\label{Eq_a2n2b}
\sum_{\substack{\la \\[1pt] \la_1\leq m}} & 
q^{\abs{\la}} P'_{2\la}\big(q^{1/2},x_2^{\pm},\dots,x_n^{\pm};q\big) \\[-1mm]
&=\frac{1}{(q)_{\infty}^n \prod_{i=1}^n  \theta(-q^{1/2}\hat{x}_i;q) 
\theta(\hat{x}_i^2;q^2) 
\prod_{1\leq i<j\leq n} \theta(\hat{x}_i/\hat{x}_j,\hat{x}_i\hat{x}_j;q)} 
\notag \\[1mm] 
&\qquad \times
\sum_{r\in\Z^n} 
\Delta_{\C}(\hat{x} q^r)
\prod_{i=1}^n (-1)^{r_i} \hat{x}_i^{\kappa r_i-i+1}
q^{\frac{1}{2}\kappa r_i^2-nr_i}, \notag 
\end{align}
where now $\hat{x}=(q^{1/2},x_2,\dots,x_n)$.
Again we are at the point where we can specialize, letting
\begin{equation}\label{Eq_A2n2b-spec}
q\mapsto q^{2n-1},\qquad x_i\mapsto q^{n-i+1/2}=:\hat{x_i}\;\; (1\leq i\leq n).
\end{equation}
This is consistent, since $x_1=q^{1/2}\mapsto q^{n-1/2}$. By
the identity 
\begin{multline*}
(q^{2n-1};q^{2n-1})_{\infty}^n 
\prod_{i=1}^n \theta(-q^{2n-i};q^{2n-1}) 
\theta(q^{2n-2i+1};q^{4n-2})  \\ \times
\prod_{1\leq i<j\leq n} \theta(q^{j-i},q^{2n-i-j+1};q^{2n-1})
=2(q)_{\infty}^n,
\end{multline*}
we obtain
\[
\sum_{\substack{\la \\[1pt] \la_1\leq m}} 
q^{2\abs{\la}} P_{2\la}\big(1,q,q^2,\dots;q^{2n-1}\big) \\
=\frac{\M}{2(q)_{\infty}^n},
\]
where
\[
\M:=\sum_{r\in\Z^n} \Delta_{\C}(\hat{x} q^{(2n-1)r})
\prod_{i=1}^n (-1)^{r_i} \hat{x}_i^{\kappa r_i-i+1}
q^{\frac{1}{2}(2n-1)\kappa r_i^2-(2n-1)nr_i}.
\]
Expressing $\M$ in determinantal form using
\eqref{Eq_detC} yields
\begin{multline*}
\M=\det_{1\leq i,j\leq n} \bigg(\sum_{r\in\Z} (-1)^r \hat{x}_i^{\kappa r-i+1}
q^{\frac{1}{2}(2n-1)\kappa r^2-(2n-1)nr} \\ \times
\Big( (\hat{x}_i q^{(2n-1)r})^{j-1}-(\hat{x}_iq^{(2n-1)r})^{2n-j+1}\Big)\bigg).
\end{multline*}
We now replace $(i,j)\mapsto (j,i)$ and, viewing the resulting determinant 
as of the form $\det\big(\sum_r u_{ij;r}-\sum_r v_{ij;r}\big)$, we change 
the summation index $r\mapsto -r$ in the sum over $u_{ij;r}$. 
The expression for $\M$ we obtain is exactly \eqref{Eq_same} 
except that $(-1)^{i-j} q^{b_{ij}}$ is replaced by $q^{c_{ij}}$
and $y_i$ is given by $q^{n-i+1}$ instead of $q^{(\kappa+1)/2-i}$.
Following the previous proof results in \eqref{Eq_RR-A2n2b}. 
\end{proof}

\medskip

To interpret \eqref{Eq_RR-A2n2b} in terms of $\A_{2n}^{(2)}$,
we note that by \cite[Lemma 2.2]{BW13} the right-hand side of 
\eqref{Eq_a2n2b} in which $\hat{x}$ is interpreted as 
\[
\hat{x}_i=-q^{1/2} \eup^{\alpha_0+\cdots+\alpha_{i-1}}\;\;(1\leq i\leq n)
\]
(and $q$ again as \eqref{Eq_null}) corresponds to the $\A_{2n}^{(2)}$ 
character 
\[
\eup^{-2m\La_0} \ch V(2m\La_0).
\]
The specialization \eqref{Eq_A2n2b-spec} is then again consistent
with \eqref{Eq_specA}. From \eqref{Eq_not-principal} with $\la=(2m,0^n)$,
the first product-form on the right of \eqref{Eq_RR-A2n2b} 
immediately follows.
By level-rank duality, we can also identify \eqref{Eq_RR-A2n2b} 
as a specialization of the $\A_{2m}^{(2)}$ character 
$\eup^{-2n\La_0} \ch V(2n\La_0)$.

\subsection{Proof of Theorem~\ref{Thm_Main3}}

This proof, which uses the $\D_n^{(1)}$ Macdonald identity,
is the most complicated of the four.

\begin{proof}[Proof of Theorem~\ref{Thm_Main3}]
Once again we iterate \eqref{Eq_key}, but now both
$x_{n-1}$ and $x_n$ remain unpaired:
\[
\lim_{y_1\to x_1^{-1}}\dots \lim_{y_{n-2}\to x_{n-2}^{-1}}
L^{(0)}_m(x_1,y_1,\dots,x_{n-2},y_{n-2},x_{n-1},x_n) 
=L^{(n-2)}_m(x_1,\dots,x_n).
\]
Accordingly, if we replace 
$x\mapsto(x_1,y_1,\dots,x_{n-2},y_{n-2},x_{n-1},x_n)$ 
in \eqref{Eq_Cn-RS-m} (thereby changing the rank from $n$ to $2n-2$)
and take the $y_i\to x_i^{-1}$ limit,
for $1\leq i\leq n-2$, we obtain
\begin{align*}
&\sum_{\substack{\la \\[1pt] \la_1\leq m}} q^{\abs{\la}} 
P'_{2\la}\big(x_1^{\pm},\dots,x_{n-2}^{\pm},x_{n-1},x_n;q\big) \\[-1mm]
&\quad=\frac{1}{(q)_{\infty}^{n-2}(qx_{n-1}^2,qx_{n-1}x_n,qx_n^2)_{\infty}} \\[1mm]
& \qquad \times 
\frac{1}{\prod_{i=1}^{n-2} 
(qx_i^{\pm 2},qx_i^{\pm}x_{n-1},qx_i^{\pm}x_n)_{\infty}
\prod_{1\leq i<j\leq n-2} (qx_i^{\pm} x_j^{\pm})_{\infty}} \\[1mm] 
&\qquad \times
\sum_{r\in\Z^n} \frac{\Delta_{\C}(x q^r)}{\Delta_{\C}(x)}\,
\prod_{i=1}^n \bigg(\frac{x_i^{\kappa}}{x_{n-1}x_n}\bigg)^{r_i}
q^{\frac{1}{2}\kappa r_i^2-(n-1)r_i}
\frac{(x_i x_{n-1},x_i x_n)_{r_i}}{(qx_i/x_{n-1},qx_i/x_n)_{r_i}},
\end{align*}
where $\kappa=2m+2n$. It is important to note that the summand vanishes unless
$r_{n-1}$ and $r_n$ are both nonnegative.
Next we let $(x_{n-1},x_n)$ tend to $(q^{1/2},1)$ using
\[
\lim_{(x_{n-1},x_n)\to (q^{1/2},1)} 
\frac{\Delta_{\C}(x q^r)}{\Delta_{\C}(x)}\,
\prod_{i=1}^n 
\frac{(x_i x_{n-1},x_i x_n)_{r_i}}{(qx_i/x_{n-1},qx_i/x_n)_{r_i}}
=\phi_{r_n} 
\frac{\Delta_{\B}(\hat{x} q^r)}{\Delta_{\B}(\hat{x})},
\]
with $\phi_r$ as in \eqref{Eq_phi} and 
$\hat{x}:=(-x_1,\dots,-x_{n-2},-q^{1/2},-1)$.
Hence we find that
\begin{align*}
&\sum_{\substack{\la \\[1pt] \la_1\leq m}} 
q^{\abs{\la}} 
P'_{2\la}\big(x_1^{\pm},\dots,x_{n-2}^{\pm},q^{1/2},1;q\big) \\[-1mm]
&\quad=\frac{1}{(q)_{\infty}^{n-1}(q^{3/2};q^{1/2})_{\infty} 
\prod_{i=1}^{n-2}  (qx_i^{\pm};q^{1/2})_{\infty}
(qx_i^{\pm 2})_{\infty}
\prod_{1\leq i<j\leq n-2} (qx_i^{\pm} x_j^{\pm})_{\infty}} \\[1mm] 
&\qquad \times
\sum_{r_1,\dots,r_{n-2}=-\infty}^{\infty}
\sum_{r_{n-1},r_n=0}^{\infty} \phi_{r_n}
\frac{\Delta_{\B}(\hat{x} q^r)}{\Delta_{\B}(\hat{x})}\,
\prod_{i=1}^n \hat{x}_i^{\kappa r_i}
q^{\frac{1}{2}\kappa r_i^2-\frac{1}{2}(2n-1)r_i}.
\end{align*}
Since the summand (without the factor $\phi_{r_n}$)
is invariant under the variable change
$r_n\mapsto -r_n$, as well as the change 
$r_{n-1}\mapsto -r_{n-1}-1$, we can rewrite this as
\begin{align*}
\sum_{\substack{\la \\[1pt] \la_1\leq m}} &
q^{\abs{\la}} 
P'_{2\la}\big(x_1^{\pm},\dots,x_{n-2}^{\pm},q^{1/2},1;q\big) \\
&=\frac{1}{(q)_{\infty}^{n-1}(q^{1/2};q^{1/2})_{\infty}
\prod_{i=1}^n \theta(\hat{x}_i;q^{1/2})
\prod_{1\leq i<j\leq n} \theta(\hat{x}_i/\hat{x}_j,\hat{x}_i\hat{x}_j)} 
\notag \\[1mm] 
&\qquad \times
\sum_{r\in\Z^n}
\Delta_{\B}(\hat{x} q^r) \prod_{i=1}^n \hat{x}_i^{\kappa r_i-i+1}
q^{\frac{1}{2}\kappa r_i^2-\frac{1}{2}(2n-1)r_i}, \notag 
\end{align*}
where, once again, we have used \eqref{Eq_simp} to clean up the
infinite products.
Before we can carry out the usual specialization, we need to
relabel $x_1,\dots,x_{n-2}$ as $x_2,\dots,x_{n-1}$ and, 
accordingly, we redefine $\hat{x}$ as $(-q^{1/2},-x_2,\dots,-x_{n-1},-1)$. 
For $n\geq 2$, we then find that
\begin{align}\label{Eq_dn}
\sum_{\substack{\la \\[1pt] \la_1\leq m}} & q^{\abs{\la}} 
P'_{2\la}\big(q^{1/2},x_2^{\pm},\dots,x_{n-1}^{\pm},1;q\big) \\
&=\frac{1}{(q)_{\infty}^{n-1}(q^{1/2};q^{1/2})_{\infty}
\prod_{i=1}^n \theta(\hat{x}_i;q^{1/2})
\prod_{1\leq i<j\leq n} \theta(\hat{x}_i/\hat{x}_j,\hat{x}_i\hat{x}_j)} 
\notag \\[1mm] 
&\qquad \times
\sum_{r\in\Z^n}
\Delta_{\B}(\hat{x} q^r) \prod_{i=1}^n \hat{x}_i^{\kappa r_i-i+1}
q^{\frac{1}{2}\kappa r_i^2-\frac{1}{2}(2n-1)r_i}. \notag 
\end{align}
We are now ready to make the substitutions
\begin{equation}\label{Eq_Dn-spec}
q\mapsto q^{2n-2},\qquad x_i\mapsto q^{n-i}\;\; (2\leq i\leq n-1),
\end{equation}
so that $\hat{x}_i:=-q^{n-i}$ for $1\leq i\leq n$.
By the identity 
\begin{multline*}
(q^{2n-2};q^{2n-2})_{\infty}^{n-1}(q^{n-1};q^{n-1})_{\infty}\prod_{i=1}^n 
\theta(-q^{n-i};q^{n-1}) \\ \times
\prod_{1\leq i<j\leq n} \theta(q^{j-i},q^{2n-i-j};q^{2n-2})=
4(q^2;q^2)_{\infty}(q)_{\infty}^{n-1}
\end{multline*}
and \eqref{Eq_PpP}, we obtain
\[
\sum_{\substack{\la \\[1pt] \la_1\leq m}}
q^{2\abs{\la}} P_{2\la}\big(1,q,q^2,\dots;q^{2n-3}\big) =
\frac{\M}{4(q^2;q^2)_{\infty}(q)_{\infty}^{n-1}},
\] 
where $\M$ is given by
\[
\M:=\sum_{r\in\Z^n}
\Delta_{\B}(\hat{x} q^{2(n-1)r}) \prod_{i=1}^n \hat{x}_i^{\kappa r_i-i+1}
q^{(n-1)\kappa r_i^2-(n-1)(2n-1)r_i}.
\]
By the $\B_n$ determinant \eqref{Eq_Bn-VdM}, we find that
\[
\M=\det_{1\leq i,j\leq n} \bigg(\sum_{r\in\Z} \hat{x}_i^{\kappa r-i+1}
q^{(n-1)\kappa r^2-(n-1)(2n-1)r} \cdot
\Big( \big(\hat{x}_i q^{2(n-1)r}\big)^{j-1}-
\big(\hat{x}_iq^{2(n-1)r}\big)^{2n-j}\Big)\bigg).
\]
By the same substitutions that transformed \eqref{Eq_begin} into
\eqref{Eq_end}, we obtain
\[
\M=\det_{1\leq i,j\leq n} \bigg((-1)^{i-j} q^{b_{ij}} 
\sum_{r\in\Z} y_i^{2(n-1) r-i+1} q^{2(n-1)\kappa \binom{r}{2}} 
\cdot \Big(\big(y_iq^{\kappa r}\big)^{j-1}+
\big(y_i q^{\kappa r}\big)^{2n-j-1}\Big)\bigg),
\]
where $y_i$ and $b_{ij}$ are as in \eqref{Eq_yb}.
Recalling the Weyl denominator formula for $\D_n$ \cite{Krattenthaler99}
\[
\frac{1}{2} \det_{1\leq i,j\leq n} \big(x_i^{j-1}+x_i^{2n-j-1}\big)=
\prod_{1\leq i<j\leq n} (x_i-x_j)(x_ix_j-1)=:\Delta_{\D}(x)
\]
we can rewrite $\M$ in the form
\[
\M=2\sum_{r\in\Z^n} \Delta_{\D}(xq^r)
\prod_{i=1}^n y_i^{2(n-1) r_i-i+1} q^{2(n-1)\kappa \binom{r_i}{2}}.
\]
Taking the $a_1,\dots,a_{2n-2}\to 0$, $a_{2n-1}\to 1$ and $a_{2n}\to -1$ 
limit in Gustafson's multiple $_6\psi_6$ summation for the affine
root system $\A_{2n-1}^{(2)}$ \cite{Gustafson89} leads to
the following variant of the $\D_n^{(1)}$ Macdonald 
identity\footnote{As in the $\B_n^{(1)}$ case, the actual $\D_n^{(1)}$ 
Macdonald identity contains the restriction 
$\abs{r}\equiv 0\pmod{2}$ on the sum over $r$.}
\[
\sum_{r\in\Z^n} \Delta_{\D}(xq^r) \prod_{i=1}^n 
x_i^{2(n-1)r_i-i+1} q^{2(n-1)\binom{r_i}{2}}
=2(q)_{\infty}^n 
\prod_{1\leq i<j\leq n} \theta(x_i/x_j,x_ix_j;q).
\]
This implies the claimed product form for $\M$ and completes our 
proof.
\end{proof}

\medskip

Identity \eqref{Eq_RR-Dn} has a  representation-theoretic interpretation.
By \cite[Lemma 2.4]{BW13},
the right-hand side of \eqref{Eq_dn} in which $\hat{x}$ is interpreted as
\[
\hat{x}_i=\eup^{-\alpha_i-\cdots-\alpha_n}\;\;(1\leq i\leq n)
\]
and $q$ as
\[
q=\eup^{-2\alpha_0-\cdots-2\alpha_n}
\]
yields the $\D_{n+1}^{(2)}$ character 
\[
\eup^{-2m\La_0} \ch V(2m\La_0).
\]
The specialization \eqref{Eq_Dn-spec} then corresponds to
\[
\eup^{-\alpha_0},\,\eup^{-\alpha_n}\mapsto -1 
\quad\text{and}\quad
\eup^{-\alpha_i}\mapsto q\;\;(2\leq i\leq n-1).
\]
Denoting this by $F$, we have
\[
F\big(\eup^{-\La} \ch V(\La)\big)
=\frac{(q^{\kappa};q^{\kappa})_{\infty}^n}
{(q^2;q^2)_{\infty}(q)_{\infty}^{n-1}}  
\prod_{1\leq i<j\leq n} 
\theta\big(q^{\la_i-\la_j-i+j},q^{\la_i+\la_j-i-j+2n+1};q^{\kappa}\big),
\]
where $\kappa=2n+2\la_0$ and
\[
\La=2(\la_0-\la_1)\La_0+(\la_1-\la_2)\La_1+\cdots+(\la_{n-1}-\la_n)\La_{n-1}+2\la_n\La_n,
\]
for $\la=(\la_0,\la_1,\dots,\la_n)$ a partition or half-partition 
(i.e., all $\la_i\in\Z+1/2$).
For $\la=(m,0^n)$ this agrees with \eqref{Eq_RR-Dn}.

\section{Proof of Theorem~\ref{Thm_Main4}}\label{Sec_Pf2}

For integers $k$ and $m$, where $0\leq k\leq m$, we denote the
the {\it nearly-rectangular} partition $(\underbrace{m,\dots,m}_{r \text{ times}},k)$ 
as $(m^r,k)$.
Using these partitions, we have the following ``limiting'' 
Rogers--Ramanujan-type
identities, which imply Theorem~\ref{Thm_Main4} when $k=0$ or $k=m$.

\begin{theorem}[$\A_{n-1}^{(1)}$ RR and AG identities]\label{Thm_Main5}
If $m$ and $n$ are positive integers and $0\leq k\leq m$,
then we have
\begin{multline}\label{Eq_limk}
\lim_{r\to\infty}
q^{-m\binom{r}{2}-kr} 
Q_{(m^r,k)}(1,q,q^2,\dots;q^n) \\
=\frac{(q^n;q^n)_{\infty} (q^{\kappa};q^{\kappa})_{\infty}^{n-1}}
{(q)_{\infty}^n}
\cdot \prod_{i=1}^{n-1} \theta(q^{i+k};q^{\kappa})
\cdot \prod_{1\leq i<j\leq n-1} \theta(q^{j-i};q^{\kappa}),
\end{multline}
where $\kappa=m+n$.
\end{theorem}

\smallskip
\begin{remark}
A similar calculation when $k\geq m$ gives 
\begin{multline*}
\lim_{r\to\infty}
q^{-m\binom{r+1}{2}} 
Q_{(k,m^r)}(1,q,q^2,\dots;q^n) \\
=\qbin{k-m+n-1}{n-1}_q
\frac{(q^n;q^n)_{\infty} (q^{\kappa};q^{\kappa})_{\infty}^{n-1}}
{(q)_{\infty}^n}
\prod_{1\leq i<j\leq n} \theta(q^{j-i};q^{\kappa}).
\end{multline*}
\end{remark}

\begin{proof}[Proof of Theorem~\ref{Thm_Main5}]
It suffices to prove the identity for $0\leq k<m$, and below we assume that $k$
satisfies this inequality.

The following identity for modified Hall--Littlewood polynomials
indexed by near-rectangular partitions is a special case of 
\cite[Corollary 3.2]{BW13}:
\begin{multline*}
Q'_{(m^r,k)}(x;q)=(q)_r(q)_1
\sum_{\substack{u\in\NN^n\\ \abs{u}=r+1}}
\sum_{\substack{v\in\NN^n\\ \abs{v}=r}}
\prod_{i=1}^n x_i^{ku_i+(m-k)v_i} q^{k\binom{u_i}{2}+(m-k)\binom{v_i}{2}}  
\\ \times
\prod_{i,j=1}^n 
\frac{(qx_i/x_j)_{u_i-u_j}}{(qx_i/x_j)_{u_i-v_j}}\cdot
\frac{(qx_i/x_j)_{v_i-v_j}}{(qx_i/x_j)_{v_i}}.
\end{multline*}
It is enough to compute the limit on the left-hand side of \eqref{Eq_limk}
for $r$ a multiple of $n$. Hence we replace $r$ by $nr$ in the above
expression, and then shift $u_i\mapsto u_i+r$ and
$v_i\mapsto v_i+r$, for all $1\leq i\leq n$, to obtain
\begin{multline*}
Q'_{(m^{nr},k)}(x;q)=(x_1\cdots x_n)^{mr}
q^{mn\binom{r}{2}+kr} (q)_{nr}(q)_1 \\ 
\times
\sum_{\substack{u\in\Z^n\\ \abs{u}=1}}
\sum_{\substack{v\in\Z^n\\ \abs{v}=0}}
\prod_{i=1}^n x_i^{ku_i+(m-k)v_i} 
q^{k\binom{u_i}{2}+(m-k)\binom{v_i}{2}}  
\prod_{i,j=1}^n 
\frac{(qx_i/x_j)_{u_i-u_j}}{(qx_i/x_j)_{u_i-v_j}}\cdot
\frac{(qx_i/x_j)_{v_i-v_j}}{(qx_i/x_j)_{r+v_i}}.
\end{multline*}
Since the summand vanishes unless $u_i\geq v_i$ for all $i$ and
$\abs{u}=\abs{v}+1$, it follows that $u=v+\epsilon_{\ell}$, for
some $\ell=1,\dots,n$, where $(\epsilon_{\ell})_i=\delta_{\ell i}$.
Hence we find that
\begin{multline*}
Q'_{(m^{nr},k)}(x;q)=(x_1\cdots x_n)^{mr}
q^{mn\binom{r}{2}+kr} (q)_{nr} \\ 
\times \sum_{\substack{v\in\Z^n\\ \abs{v}=0}}
\prod_{i=1}^n x_i^{mv_i} q^{m\binom{v_i}{2}} 
\prod_{i,j=1}^n \frac{(qx_i/x_j)_{v_i-v_j}}{(qx_i/x_j)_{r+v_i}}
\sum_{\ell=1}^n \big(x_{\ell}q^{v_{\ell}}\big)^k 
\prod_{\substack{i=1 \\ i\neq k}}^n \frac{1}{1-q^{v_i-v_{\ell}} x_i/x_{\ell}}.
\end{multline*}
Next we use
\[
\prod_{i,j=1}^n (qx_i/x_j)_{v_i-v_j}
=\frac{\Delta(xq^v)}{\Delta(x)}\,
(-1)^{(n-1)\abs{v}} q^{-\binom{\abs{v}}{2}}
\prod_{i=1}^n x_i^{nv_i-\abs{v}} q^{n\binom{v_i}{2}+(i-1)v_i},
\]
where $\Delta(x):=\prod_{1\leq i<j\leq n} (1-x_i/x_j)$,
and
\[
\sum_{\ell=1}^n x_{\ell}^k 
\prod_{\substack{i=1 \\ i\neq k}}^n \frac{1}{1-x_i/x_{\ell}}=
\sum_{1\leq i_1\leq i_2\leq\cdots\leq i_k\leq n} x_{i_1}x_{i_2}\cdots x_{i_k}
=h_k(x)=s_{(k)}(x),
\]
where $h_k$ and $s_{\la}$ are the complete symmetric and Schur function,
respectively. Thus we have
\begin{multline*}
Q'_{(m^{nr},k)}(x;q)=(x_1\cdots x_n)^{mr} q^{mn\binom{r}{2}+kr}(q)_{nr} \\
\times 
\sum_{\substack{v\in\Z^n\\ \abs{v}=0}} s_{(k)}(xq^v)
\frac{\Delta(xq^v)}{\Delta(x)}
\prod_{i=1}^n x_i^{\kappa v_i}q^{\frac{1}{2}\kappa v_i^2+iv_i}
\prod_{i,j=1}^n \frac{1}{(qx_i/x_j)_{r+v_i}},
\end{multline*}
where $\kappa:=m+n$.
Note that the summand vanishes unless $v_i\geq -r$ for all $i$.
This implies the limit
\begin{multline*}
\lim_{r\to\infty}
q^{-mn\binom{r}{2}-kr} \frac{Q'_{(m^{nr},k)}(x;q)}
{(x_1\cdots x_n)^{mr}} \\
=\frac{1}{(q)_{\infty}^{n-1}\prod_{1\leq i<j\leq n} \theta(x_i/x_j;q)} 
\sum_{\substack{v\in\Z^n\\ \abs{v}=0}}
s_{(k)}(xq^v)
\Delta(xq^v) \prod_{i=1}^n x_i^{\kappa v_i}q^{\frac{1}{2}\kappa v_i^2+iv_i}.
\end{multline*}
The expression on the right is exactly the 
Weyl--Kac formula for the level-$m$ $\A_{n-1}^{(1)}$ character \cite{Kac90}
\[
\eup^{-\La} \ch V(\La) , \quad \La=(m-k)\La_0+k\La_1,
\]
provided we identify
\[
q=\eup^{-\alpha_0-\alpha_1-\cdots-\alpha_{n-1}}
\quad\text{and}\quad
x_i/x_{i+1}=\eup^{-\alpha_i}\;\; (1\leq i\leq n-1).
\]
Hence 
\[
\lim_{r\to\infty}
q^{-mn\binom{r}{2}-kr} \frac{Q'_{(m^{nr},k)}(x;q)}
{(x_1\cdots x_n)^{mr}}
=\eup^{-\Lambda} \ch V(\Lambda),
\]
with $\La$ as above. 
For $m=1$ and $k=0$ this was obtained in \cite{Kirillov00} by 
more elementary means.
The simultaneous substitutions $q\mapsto q^n$ and $x_i\mapsto q^{n-i}$
correspond to the principal specialization \eqref{Eq_PS}. 
From \eqref{Eq_Kac} we can then read off the product form claimed
in \eqref{Eq_limk}.
\end{proof}

\section{Siegel Functions}\label{SiegelFunctions}
The normalizations for the series $\Phi_*$ were chosen so that the resulting $q$-series are modular functions on the congruence subgroups
$\Gamma(N)$, where
\[
\Gamma(N):=\big\{ \big(\begin{smallmatrix} a&b\\c&d\end{smallmatrix}\big) 
\in \SL_2(\Z):\: a\equiv d\equiv 1\pmod N,\  b\equiv c\equiv 0\pmod N\big\}.
\]
These groups act on $\H$, the upper-half of the complex plane,
by $\gamma \tau:=\frac{a\tau+b}{c\tau+d}$, where 
$\gamma=\big(\begin{smallmatrix} a&b\\c&d\end{smallmatrix}\big)$.
If $f$ is a meromorphic function on $\H$ and $\gamma\in \SL_2(Z)$, then we 
define
\[
(f |_k \gamma )(\tau):=(c\tau+d)^{-k} f(\gamma \tau).
\]
Modular functions are meromorphic functions which are invariant with 
respect to this action. More precisely, a meromorphic function $f$ on $\H$ is 
a \textit{modular function} on $\Gamma(N)$ if for every $\gamma\in\Gamma(N)$ 
we have
\[
f(\gamma \tau)=(f|_0 \gamma)(\tau)=f(\tau).
\]
The set of such functions forms a field.
We let $\mathcal{F}_N$ denote the canonical subfield of those modular 
functions on $\Gamma(N)$ whose Fourier expansions are defined over 
$\Rat(\zeta_N)$, where $\zeta_N:=\eup^{2\pi\iup/N}$.

The important work of Kubert and Lang \cite{KubertLang} 
plays a central role in the study of these modular function fields. 
Their work, which is built around the Siegel $g_a$ functions and the Klein 
$\mathfrak{t}_a$ functions, allows us to understand the fields 
$\mathcal F_N$, as well as the Galois theoretic properties of the extensions
$\mathcal F_N/\mathcal F_1$. 
These results will be fundamental tools in the proofs of Theorems~\ref{thm} 
and \ref{thm2}.

\subsection{Basic Facts about Siegel functions}

We begin by recalling the definitions of the Siegel and Klein functions.
Let $\mathbf{B}_2(x):=x^2-x+\frac16$ be the second Bernoulli 
polynomial and $\eup(x):=\eup^{2\pi\iup x}$.
If $a=(a_1,a_2)\in \Rat^2$, then the \textit{Siegel function} $g_a$ is defined as 
\begin{equation}\label{Siegel_ga}
g_a(\tau):=q^{\tfrac{1}{2} \mathbf{B}_2(a_1)}
\eup\big(a_2(a_1-1)/2\big)
\prod_{n=1}^{\infty} \big(1-q^{n-1+a_1}\eup(a_2)\big)
\big(1-q^{n-a_1}\eup(-a_2)\big).
\end{equation}
Notice this is a change of sign from the usual normalization of the Siegel function.
The \textit{Klein function} $\mathfrak t_a$ is defined as 
\begin{equation}
\mathfrak t_a(\tau):= -\frac{g_a(\tau)}{\eta(\tau)^2}
\end{equation}
where $\eta(\tau):=q^{1/24}\prod_{n=1}^{\infty}(1-q^n)$ is the 
\textit{Dedekind $\eta$-function}.

Neither $g_a$ nor $\mathfrak t_a$ are modular on $\Gamma(N)$, 
however if $N\cdot a\in \Z^2$, then $\mathfrak t_a^{2N}$ is on $\Gamma(N)$ 
(or $\mathfrak t_a^N$ if $N$ is odd). 
Therefore if $g_a^{\lcm(12,2N)}\in \mathcal F_N,$ and if 
$N\cdot a'\in \Z^2$, then 
$\big(\frac{g_a(\tau)}{g_{a'}(\tau)}\big)^{2N}\in \mathcal F_N$ if $N$ 
is even and 
$\big(\frac{g_a(\tau)}{g_{a'}(\tau)}\big)^N\in \mathcal F_N$ if $N$ is odd.
Given $a\in \Rat^2$, we denote the smallest $N\in \N$ such that 
$N\cdot a\in \Z^2$ by $\Den(a)$.

\begin{theorem}[{\cite[Ch.~2 of K1 and K2]{KubertLang}}]
Assuming the notation above, the following are true:
\begin{enumerate}
\item If $\gamma \in \SL_2(\Z)$, then
\[
(\mathfrak{t}_a |_{-1}\gamma)(\tau)=\mathfrak{t}_{a\gamma}(\tau).
\]

\item If $b=(b_1,b_2)\in\Z^2$, then
\[
\mathfrak{t}_{a+b}(\tau)=\eup\big(1/2\cdot(b_1b_2+b_1+b_2-b_1a_2+b_2a_1)\big)
\mathfrak{t}_{a\gamma}(\tau).
\]
\end{enumerate}
\end{theorem}
These properties for $\mathfrak{t}_a$, \eqref{Siegel_ga}, 
and the fact that $\eta(\tau)^{24}=\Delta(\tau)$ is modular on $\SL_2(\Z)$, 
lead to the following properties for $g_a$.
\begin{theorem}[{\cite[ Ch.~2, Thm~1.2]{KubertLang}}]\label{KubertLangThm}
If $a\in\Z^2/N$ and $\Den(a)=N$, then the following are true:
\begin{enumerate}
\item If $\gamma\in \SL_2(\Z)$, then
\[
(g_a^{12}|_0\gamma)(\tau)=g^{12}_{a\gamma}(\tau).
\]

\item If $b=(b_1,b_2)\in\Z^2$, then
\[
g_{a+b}(\tau)=\eup\big(1/2\cdot(b_1b_2+b_1+b_2-b_1a_2+b_2a_1)\big)g_a(\tau).
\]
\item We have that $g_{-a}(\tau)=-g_{a}(\tau)$.
\item The $g_a(\tau)^{12N}$ are modular functions on $\Gamma(N)$. 
Moreover, if $\gamma\in \SL_2(\Z)$, then we have
\[
(g_a^{12}|_0\gamma)(\tau)=g^{12}_{a\gamma}(\tau).
\]
\end{enumerate}
\end{theorem}

The following theorem addresses the modularity properties of 
products and quotients of Siegel functions.

\begin{theorem}[{\cite[Ch.~3, Lemma 5.2, Thm 5.3]{KubertLang}}]\label{F_N}
Let $N\geq 2$ be an integer, and let $\{m(a)\}_{r\in \frac{1}{N}\Z^2/\Z^2}$ 
be a set of integers. 
Then the product of Siegel functions
\[
\prod_{a\in \frac{1}{N}\Z^2/\Z^2}g_a^{m(a)}(\tau)
\]
belongs to $\mathcal F_N$ if $\{m(a)\}$ satisfies the following:
\begin{enumerate}
\item We have that $\sum_a m(a)(Na_1)^2\equiv 
\sum_a m(a)(Na_2)^2\equiv 0 \pmod{\gcd(2,N)\cdot N}$.
\item We have that $\sum_a m(a)(Na_1)(Na_2)\equiv 0 \pmod{N}$.
\item We have that $\gcd(12,N)\cdot \sum_a m(a)\equiv 0 \pmod{12}$.
\end{enumerate}
\end{theorem}

Additionally, we have the following important results about the 
algebraicity of the singular values of the Siegel functions in relation to
the singular values of the $\SL_2(\Z)$ modular function
\begin{align*}
j(\tau)&:=\frac{\big(1+240\sum_{n=1}^{\infty} 
\sum_{d\mid n}d^3 q^n\big)^3}{q\prod_{n=1}^{\infty}
(1-q^n)^{24}}\\
&\hphantom{:}=
\frac{\eta(\tau)^{24}}{\eta(2\tau)^{24}}+3\cdot2^{8}+3\cdot2^{16}\,
\frac{\eta(2\tau)^{24}}{\eta(\tau)^{24}}+2^{24}\,
\frac{\eta(2\tau)^{48}}{\eta(\tau)^{48}}\\[2mm]
&\hphantom{:}=q^{-1}+744+196884q+\cdots,
\end{align*}
\noindent
which are well known to be algebraic by the theory of 
complex multiplication (for example, see \cite{Borel, Cox}).

\begin{theorem}[{\cite[Ch. 1, Thm. 2.2]{KubertLang}}]\label{ga-Algebraic} 
If $\tau$ is a CM point and $N=\Den(a)$, then the following are true:
\begin{enumerate}
\item We have that $g_a(\tau)$ is an algebraic integer.
\item If $N$ has at least two prime factors, then $g_a(\tau)$ is a 
unit over $\Z[j(\tau)]$.
\item If $N=p^r$ is a prime power, then $g_a(\tau)$ is a unit over 
$\Z[1/p][j(\tau)]$.
\item If $c\in \Z$ and $(c,N)=1$, then $(g_{ca}/g_a)$ is a unit over 
$\Z[j(\tau)]$.
\end{enumerate}
\end{theorem}

\subsection{Galois theory of singular values of products of Siegel 
functions}\label{Galois}
We now recall the Galois-theoretic properties of extensions of modular 
function fields, and we then relate these properties to the Siegel and 
Klein functions.

The Galois group $\Gal(\mathcal{F}_N/\mathcal{F}_1)$ is isomorphic to 
$\GL_2(N)/\{\pm I\}=\GL_2(\Z/N\Z)/{\{\pm I}\}$
(see {\cite[Ch. 3, Lemma 2.1]{KubertLang}}), where $I$ is the identity matrix.
This group factors naturally as 
\[
\bigg\{ \begin{pmatrix}1&0\\0&d\end{pmatrix}~:~ d\in (\Z/N\Z)^\times\bigg\}
\times \SL_2(N)/{\{\pm I\}},
\]
where an element $\left(\begin{smallmatrix}1&0\\0&d\end{smallmatrix}\right)$ 
acts on the Fourier coefficients by sending $\zeta_N\to\zeta_N^d$, and a 
matrix $\gamma\in \SL_2(\Z)$ acts by the standard fractional linear 
transformation on $\tau$. 
If $f(\tau)\in \mathcal{F}_N$ and $\gamma\in \GL_2(N)$, then we use the 
notation $f(\tau)_{(\gamma)}:=(\gamma \circ f)(\tau)$. 
Applying these facts to the Siegel functions, we obtain the following.

\begin{proposition}
If $a\in \Rat^2$, and $\Den(a)$ divides $N$, then the multiset  
\[
\big\{ g^{12N}_a(\tau)(\gamma):=g^{12N}_{a\gamma}(\tau):~ 
\gamma\in \GL_2(N)\big\}
\]
is a union of Galois orbits for $g^{12N}_a(\tau)$ over $\mathcal F_1$.
\end{proposition}

If $\theta$ is a CM point of discriminant $-D$, we define the field 
\[
K_{(N)}(\theta):=\Rat(\theta)\big(f(\theta):~f\in \mathcal F_N 
\textrm{ s.t. $f$ is defined and finite at $\theta$}\big),
\]
and $\mathcal H:=\Rat(\theta,j(\theta))$ be the Hilbert class field over 
$\Rat(\theta)$. 
The Galois group $K_{(N)}(\theta)/\mathcal H$ is isomorphic to the matrix 
group $W_{N,\theta}$ (see \cite{Shimura}) defined by
\[
W_{N,\theta}=\bigg\{\begin{pmatrix}t-sB&-sC\\sA&t\end{pmatrix} 
\in \GL_2(\Z/N\Z) \bigg\}/\bigg\{\pm \begin{pmatrix}1&0\\0&1\end{pmatrix} 
\bigg\},
\]
where $Ax^2+Bx+C$ is a minimal polynomial for $\theta$ over $\Z$. 
The Galois group $\Gal(\mathcal H/ \Rat)$ is isomorphic to the group 
$\mathcal Q_D$ of primitive reduced positive-definite integer binary 
quadratic forms of negative discriminant $-D$. 
For each $Q=ax^2+bxy+cy^2\in \mathcal Q_D$, we define the corresponding 
CM point $\tau_Q=\frac{-b+\sqrt{-D}}{2a}$. 
In order to define the action of this group, we must also define 
corresponding matrices $\beta_Q\in \GL_2(\Z/N \Z)$ which we may build up 
by way of the Chinese Remainder Theorem and the following congruences. 
For each prime $p$ dividing $N$, we have the following congruences which 
hold $\pmod {p^{\ord_p(N)}}$:
\begin{align*}
\beta_Q & \equiv \begin{cases} 
\small \begin{pmatrix}a&\frac b2\\0&1\end{pmatrix} & 
\text{if $p\nmid a$}\\
\small \begin{pmatrix}-\frac b2&-c\\1&0\end{pmatrix} & 
\text{if $p|a$, and $p\nmid c$} \\
\small \begin{pmatrix}-\frac b2 -a&-\frac b2-c\\1&-1\end{pmatrix} & 
\text{if  $p|a$, and $p|c$}
\end{cases} 
\intertext{if $-D\equiv 0 \pmod 4$, and}
\beta_Q & \equiv \begin{cases} 
\small \begin{pmatrix}a&\frac {b-1}{2}\\0&1\end{pmatrix} & 
\text{if $p\nmid a$} \\
\small \begin{pmatrix}-\frac {b+1}{2}&-c\\1&0\end{pmatrix} & 
\text{if $p|a$,\textrm{ and } $p\nmid c$} \\
\small \begin{pmatrix}-\frac {b+1}{2} -a&\frac {1-b}{2}-c\\1&-1\end{pmatrix} & 
\text{if $p|a$, \textrm{ and }$p|c$} \end{cases} 
\end{align*}
if $- D\equiv 1 \pmod 4$. 
Then given $\theta=\tau_{Q'}$ for some $Q'\in \mathcal Q_D$, define $\delta_{Q}(\theta):=\beta_{Q'}^{-1}\beta_Q$.
The Galois group $\Gal (\mathcal H/ \Rat)$ can be extended into 
$\Gal(K_{(N)}(\theta)/\Rat)$ by taking the action of a quadratic 
form $Q$ on the element $f(\theta)\in K_{(N)}(\theta)$ to be given by
\[
Q\circ f(\theta)=f(\tau_Q)_{(\delta_{Q}(\theta))}.
\]
We combine these facts into the following theorem.
\begin{theorem}\label{GaloisThm}
Let $F(\tau)$ be in $\mathcal F_N$ and let $\theta$ be a CM point of 
discriminant $-D<0$. Then the multiset
\[
\big\{F(\tau_Q)_{(\gamma\cdot \delta_{Q}(\theta))}:\:(\gamma,Q)\in 
W_{\kappa,\tau}\times \mathcal{Q}_D\big\}
\]
is a union of the Galois orbits  of $F(\theta)$ over $\Rat$.
\end{theorem}

\section{Proofs of Theorems~\ref{thm} and \ref{thm2}}\label{Proofs}
Here we prove Theorems~\ref{thm} and \ref{thm2}. 
We shall prove these theorems using the results of the previous section.

\subsection{Reformulation of the $\Phi_*(m,n;\tau)$ series}

To ease the proofs of Theorems~\ref{thm} and ~\ref{thm2}, we begin
by reformulating each of the $\Phi_*(m,n;\tau)$ series, as well as
\[
\frac{\Phi_{1a}(m,n;\tau)}{\Phi_{1b}(m,n;\tau)},
\]
as pure products of modified theta functions.
These factorizations will be more useful for our purposes. 
In order to ease notation, for a fixed $\kappa$, if $1\leq j<\kappa/2$, 
then we let
\[
\theta_{j,\kappa}:=\theta(q^j;q^\kappa),
\]
If $\kappa$ is even, then we let
\[
\theta_{\kappa/2,\kappa}:=
(q^{\kappa/2};q^\kappa)_\infty=\theta(q^{\kappa/2};q^{2\kappa}),
\]
which is a square root of $\theta(q^{\kappa/2};q^{\kappa}).$

The reformulations below follow directly from \eqref{Phi123}
by making use of the fact that
\[
\frac{(q^\kappa;q^\kappa)_\infty}{(q)_\infty}
=\prod_{j=1}^{\lfloor\kappa/2\rfloor}\theta_{j,\kappa}.
\]

\begin{lemma}\label{Phis.2}
Let $m$ and $n$ be positive integers and $\kappa_{\ast}=\kappa_{\ast}(m,n)$ 
as in \eqref{Eq_kappa}. Then the following are true:
\begin{enumerate}
\item[(1a)] With $\kappa=\kappa_1$,
\[
\Phi_{1a}(m,n;\tau)=q^{\frac{mn(4mn-4m+2n-3)}{12\kappa}}
\prod_{j=1}^m\theta_{j,\kappa}^{-1}
\prod_{j=1}^{m+n}\theta_{j,\kappa}^{-\min(m,n-1,\lceil j/2\rceil-1)}.
\]
\item[(1b)] With $\kappa=\kappa_1$, 
\[
\Phi_{1b}(m,n;\tau)=q^{\frac{mn(4mn+2m+2n+3)}{12\kappa}}\prod_{j=1}^{m+n}
\theta_{j,\kappa}^{-\min(m,n,\lfloor j/2\rfloor)}.
\]

\item[(2)] With $\kappa=\kappa_2$,
\[
\Phi_2(m,n;\tau)=
q^{\frac{m(2n+1)(2mn-m+n-1)}{12\kappa}}\prod_{j=1}^m\theta_{j,\kappa}^{-1}
\prod_{j=1}^{m+n+1}\theta_{j,\kappa}^{-\min(m,n-1,\lceil j/2\rceil-1)}
\prod_{j=n}^{\lfloor(m+n)/2\rfloor}\theta_{2j+1,\kappa}^{-1}.
\]

\item[(3)] For $n\geq 2$ and $\kappa=\kappa_3$,
\[
\Phi_3(m,n;\tau)=
q^{\frac{m(2n-1)(2mn+n+1)}{12\kappa}} 
\prod_{j=1}^m \theta(q^{2j};q^\kappa)^{-1}
\prod_{j=1}^{m+n}\theta_{j,\kappa}^{-\min(m,n-2,\lceil j/2\rceil-1)}
\prod_{j=n}^{\lfloor(m+n+1)/2\rfloor}\theta_{2j-1,\kappa}^{-1}.
\]
\end{enumerate}
Moreover, with $\kappa=\kappa_1(m,n)$, 
\[
\Psi_1(m,n;\tau):=\frac{\Phi_{1a}(m,n;\tau)}{\Phi_{1b}(m,n;\tau)}
=
q^{-\frac{mn(m+1)}{2\kappa}}
\prod_{j=1}^m\frac{\theta(q^{2j};q^\kappa)}{\theta(q^j;q^\kappa)},
\]
and with $\kappa=\kappa_2(m,n)=\kappa_3(m,n+1)$, 
\[
\Psi_2(m,n;\tau):=\frac{\Phi_2(m,n;\tau)}{\Phi_3(m,n+1;\tau)}
=
q^{-\frac{m(m+1)(2n+1)}{4\kappa}}
\prod_{j=1}^m\frac{\theta(q^{2j};q^\kappa)}{\theta(q^j;q^\kappa)}.
\]
\end{lemma}

\begin{proof}
Since the proofs of the four cases are essentially the same,
we only prove Lemma~\ref{Phis.2}~(1a).

Let $\varphi=mn(4mn-4m+2n-3)/(12\kappa_1)$.
By Theorem~\ref{Thm_Main}, we have that
\begin{align*}
\Phi_{1a}(m,n;\tau)&=q^{\varphi}\cdot 
\frac{(q^{\kappa};q^{\kappa})_{\infty}^n}{(q)_{\infty}^n}
\prod_{i=1}^n\theta(q^{i+m};q^\kappa)\prod_{1\leq i<j\leq n}
\theta(q^{j-i},q^{i+j-1};q^{\kappa})\\
&=q^{\varphi}\cdot 
\frac{(q^\kappa;q^\kappa)_{\infty}^m}{(q)_{\infty}^m}
\prod_{i=1}^{m}\theta(q^{i+1};q^{\kappa})\prod_{1\leq i<j\leq m}
\theta(q^{j-i},q^{i+j+1};q^{\kappa}).
\end{align*}
Using the simple identity
\[
\frac{(q^\kappa;q^\kappa)_{\infty}}{(q)_{\infty}}
=\prod_{j=1}^{m+n}\theta(q^j;q^{\kappa})^{-1}, 
\]
we can rewrite these two forms as 
\begin{align*}
\Phi_{1a}(m,n;\tau)&=q^{\varphi}\cdot \frac{\prod_{i=1}^n\theta(q^{i+m};q^\kappa)}
{\prod_{i=1}^{m+n}\theta(q^j;q^{\kappa})}\cdot 
\prod_{j=2}^n\frac{\prod_{i=1}^{j-1}\theta(q^{j-i},q^{i+j-1};q^{\kappa})}
{\prod_{i=1}^{m+n}\theta(q^i,q^{\kappa})}\\
&=q^{\varphi}\cdot \frac{\prod_{i=1}^m\theta(q^{i+1};q^{\kappa})}
{\prod_{i=1}^{m+n}\theta(q^j;q^{\kappa})}\cdot\prod_{j=2}^m
\frac{\prod_{i=1}^{j-1}\theta(q^{j-i},q^{i+j+1};q^{\kappa})}
{\prod_{i=1}^{m+n}\theta(q^i;q^{\kappa})}.
\end{align*}
If $m\geq n-1$ then the first identity reduces to 
\begin{align*}
\Phi_{1a}(m,n;\tau)&=q^{\varphi}\cdot 
\bigg(\prod_{j=1}^m \theta(q^j;q^\kappa)\bigg)^{-1} \,
\prod_{j=2}^n\bigg(\,\prod_{i=2j-1}^{m+n}\theta(q^i;q^\kappa)\bigg)^{-1}\\
&=q^{\varphi}\cdot 
\bigg(\prod_{j=1}^m\theta(q^j;q^\kappa)\bigg)^{-1} \,
\prod_{j=1}^{n-1}\bigg(\,\prod_{i=2j+1}^{m+n}\theta(q^i;q^\kappa)\bigg)^{-1}.
\end{align*}
If $m\leq n-1$ then the second identity reduces to
\begin{align*}
\Phi_{1a}(m,n;\tau)&=q^{\varphi}\cdot \bigg(\theta(q;q^{\kappa})
\prod_{j=m+2}^{m+n}\theta(q^j;q^{\kappa})\bigg)^{-1}\\
&\qquad\qquad \times
\prod_{j=2}^m\bigg(\theta(q^j,q^{j+1};q^{\kappa})
\prod_{i=2j+1}^{m+n}\theta(q^i;q^{\kappa})\bigg)^{-1}\\
&=q^{\varphi}\cdot \bigg(\prod_{j=1}^m\theta(q^j;q^{\kappa})\bigg)^{-1}
\prod_{j=1}^m\bigg(\,\prod_{i=2j+1}^{m+n}\theta(q^i;q^{\kappa})\bigg)^{-1}.
\end{align*}
Together these imply Lemma~\ref{Phis.2}~(1a). 
\end{proof}

Since the modified $\theta$-functions $\theta(q^\ell;q^\kappa)$ are 
essentially Siegel functions (up to powers of $q$), we can immediately 
rewrite Lemma~\ref{Phis.2} in terms of modular functions. 
We shall omit the proofs for brevity.

\begin{lemma}\label{Phis3}
Let $m$ and $n$ be positive integers and $\kappa_{\ast}=\kappa_{\ast}(m,n)$ 
as in \eqref{Eq_kappa}. Then the following are true:
\begin{enumerate}
\item[(1a)] With $\kappa=\kappa_1$,
\[
\Phi_{1a}(m,n;\tau)=\prod_{j=1}^m
g_{j/\kappa,0}(\kappa \tau)^{-1}\prod_{j=1}^{m+n}
g_{j/\kappa,0}(\kappa \tau)^{-\min(m,n-1,\lceil j/2\rceil-1)}.
\]

\item[(1b)] With $\kappa=\kappa_1$,
\[
\Phi_{1b}(m,n;\tau)=\prod_{j=1}^{m+n}
g_{j/\kappa,0}(\kappa \tau)^{-\min(m,n,\lfloor j/2\rfloor)}.
\]

\item[(2)] With $\kappa=\kappa_2$,
\begin{multline*}
\qquad\qquad \Phi_2(m,n;\tau)=
g_{\frac{1}{4},0}(2\kappa \tau)^{-\min(m,n-1)-\delta}
\prod_{j=1}^m g_{\frac j\kappa,0}(\kappa \tau)^{-1}\\
\times \prod_{j=1}^{m+n}
g_{\frac{j}{\kappa},0}(\kappa\tau)^{-\min(m,n-1,\lceil j/2\rceil-1)}
\prod_{j=n}^{\lfloor(m+n-1)/2\rfloor}
g_{\frac{(2j+1)}{\kappa},0}(\kappa \tau)^{-1}.
\end{multline*}

\item[(3)] For $n\geq 2$ and $\kappa=\kappa_3$,
\begin{multline*}
\qquad\qquad \Phi_3(m,n;\tau)=
g_{\frac{1}{4},0}(2\kappa \tau)^{-\min(m,n-2)-\delta}
\prod_{j=1}^m g_{\frac {2j}\kappa,0}(\kappa \tau)^{-1}\\
\times \prod_{j=1}^{m+n-1}
g_{\frac{j}{\kappa},0}(\kappa\tau)^{-\min(m,n-2,\lceil j/2\rceil-1)}
\prod_{j=n}^{\lfloor(m+n)/2\rfloor}
g_{\frac{(2j-1)}{\kappa},0}(\kappa \tau)^{-1}.
\end{multline*}

\item[(4)] With $\kappa=\kappa_1$,
\[
\Psi_1(m,n;\tau):=\frac{\Phi_{1a}(m,n;\tau)}{\Phi_{1b}(m,n;\tau)}
=\prod_{j=1}^m \frac{g_{\frac{2j}{\kappa},0}(\kappa \tau)}
{g_{\frac {j}{\kappa},0}(\kappa \tau)}.
\]

\item[(5)]
With $\kappa=\kappa_2(m,n)=\kappa_3(m,n+1)$, 
\[
\Psi_2(m,n;\tau):=\frac{\Phi_2(m,n;\tau)}{\Phi_3(m,n+1;\tau)}
=\prod_{j=1}^m\frac{g_{\frac{2j}{\kappa},0}(\kappa \tau)}
{g_{\frac{j}{\kappa},0}(\kappa \tau)}.
\]
\end{enumerate}
In parts (2) and (3) we have $\delta=0$ or $1$ depending on if $\kappa/2$ is even or odd respectively.
\end{lemma}

\subsection{Proofs of Theorems~\ref{thm} and ~\ref{thm2}}

We now apply the results in Section~\ref{SiegelFunctions} to prove 
Theorems~\ref{thm} and \ref{thm2}.

\begin{proof}[Proof of Theorem~\ref{thm} (1) and (2)]
Lemma~\ref{Phis3} shows that each of the $\Phi_*(m,n;\tau)$ 
is exactly a pure product of Siegel functions. 
Therefore, we may apply Theorem~\ref{KubertLangThm}
directly to each of the Siegel function factors, and as a consequence 
to each $\Phi_{*}(m,n;\tau)$.

Since by Theorem~\ref{KubertLangThm}~(4), $g_a(\tau)^{12N}$ is in 
$\mathcal F_N$ if $N=\Den(a)$, we may take $N=\kappa_{*}(m,n)$, and so 
we have that $\Phi_{*}(m,n;\tau)^{12\kappa}\in \mathcal F_{\kappa_{*}(m,n)}$.
We now apply Theorem~\ref {GaloisThm} to obtain 
Theorem~\ref{thm}~(1) and (2).
\end{proof}

\begin{proof}[Sketch of the Proof of Theorem~\ref{thm} (3)]
By Theorem~\ref{thm} (2), we have that this multiset consists of 
multiple copies of a single Galois orbit of conjugates over $\Rat$. 
Therefore to complete the proof, it suffices to show that the given 
conditions imply that there are singular values which are not repeated. 
To this end, we focus on those CM points with maximal imaginary parts. 
Indeed, because each $\Phi_{*}(m,n;\tau)$ begins with a negative power 
of $q$, one generically expects that these corresponding singular values 
will be the ones with maximal complex absolute value.

To make this argument precise requires some cumbersome but 
unenlightening details (which we omit)\footnote{A similar analysis 
is carried out in detail by Jung, Koo, and Shin in 
\cite[Sec.~4]{JungKooShin}.}.
One begins by observing why the given conditions are necessary.
For small $\kappa$ it can happen that the matrices in 
$W_{\kappa,\tau}$ permute the Siegel functions in the factorizations of
$\Phi_{*}(m,n;\tau)$ obtained in Lemma~\ref{Phis3}. 
However, if $\kappa>9$, then this does not happen. 
The condition that $\gcd(D_0,\kappa)=1$ is required for a similar reason.
More precisely, the group does not act faithfully.
However, under these conditions, the only obstruction to the conclusion 
would be a nontrivial identity between the evaluations of two different 
modular functions. In particular, under the given assumptions, we may view
these functions as a product of distinct Siegel functions. Therefore, the 
proof follows by studying the asymptotic properties of the CM values
of individual Siegel functions, and then considering the $\Phi_*$ functions as a product of these values.

The relevant asymptotics arise by considering, for each $-D$, a canonical 
CM point with discriminant $-D$. Namely, we let
\[
\tau_{*}:=\begin{cases} \frac{\sqrt{-D}}{2} &\text{if $-D\equiv 0\pmod{4}$},\\
\frac{1+\sqrt{-D}}{2} \quad &\text{if $-D\equiv 1\pmod 4$}.
\end{cases}
\]
By the theory of reduced binary quadratic forms, these points are the CM 
points with maximal imaginary parts corresponding to reduced forms with
discriminant $-D$. Moreover, every other CM point with discriminant $-D$ has 
imaginary part less than $\abs{\sqrt{-D}}/3$.
Now the singular values of each Siegel function then essentially arise 
from the values of the second Bernoulli polynomial.
The point is that one can uniformly estimate the infinite product 
portion of each singular value, and it turns out that they
are exponentially close to the number 1.  
By assembling these estimates carefully, one obtains the result.
\end{proof}

\begin{proof}[Proof of Theorem~\ref{thm2}]
Lemma~\ref{Phis3} reformulates each $\Phi_{*}$ function in terms of 
products of negative powers of Siegel functions of the form 
$g_{j/\kappa,0}(\kappa \tau)$, where $1\leq j\leq \kappa/2$,
and $g_{1/4,0}(2\kappa \tau)$, when $\kappa$ is even.
Theorem~\ref{ga-Algebraic}~(1) then implies Theorem~\ref{thm2}~(1). 

Since $\Den(j/\kappa,0)$ may be any divisor of $\kappa$, and since 
$j(\tau)$ is an algebraic integer \cite{Borel,Cox}, 
Theorem~\ref{ga-Algebraic}~(2) and (3) imply Theorem~\ref{thm2}~(2).

Using Theorem~\ref{Phis3}~(5), we have that 
\[
\frac{\Phi_{1a}(m,n;\tau)}{\Phi_{1b}(m,n;\tau)}
=\prod_{j=1}^m \frac{g_{\frac{2j}{\kappa},0}(\kappa \tau)}
{g_{\frac{j}{\kappa},0}(\kappa \tau)},
\]
where $\kappa=\kappa_1=2m+2n+1$.
Since $\kappa$ is odd, Theorem~\ref{ga-Algebraic}~(4) implies that each term 
\[
\frac{g_{\frac{2j}{\kappa},0}(\kappa\tau)}
{g_{\frac j\kappa,0}(\kappa \tau)}
\]
in the product is a unit. Therefore, Theorem~\ref{thm2}~(3) follows.
\end{proof}

\section{Examples}\label{Sec_Examples}
Here we give two examples of the main results in this paper.

\smallskip

\begin{example} 
This is a detailed discussion of the example in Section~\ref{Intro}.

Consider the $q$-series 
\begin{align*}
\Phi_{1a}(2,2;\tau)
&=q^{1/3}\prod_{n=1}^{\infty}\frac{(1-q^{9n})}{(1-q^n)}\\
&=q^{1/3}+q^{4/3}+2q^{7/3}+3q^{10/3}+5q^{13/3}+7q^{16/3}+\cdots,
\intertext{and}
\Phi_{1b}(2,2;\tau)
&=q\prod_{n=1}^\infty\frac{(1-q^{9n})(1-q^{9n-1})(1-q^{9n-8})}
{(1-q^{n})(1-q^{9n-4})(1-q^{9n-5})}\\
&=q+q^3+q^4+3q^5+3q^6+5q^7+6q^8+\cdots
\end{align*}
For $\tau=\iup/3$, the first 100 coefficients of the $q$-series
respectively give the numerical approximations
\begin{align*}
\Phi_{1a}(2,2;\iup/3)&=0.577350\dots\stackrel{?}{=}\frac{1}{\sqrt{3}}\\
\Phi_{1b}(2,2;\iup/3)&=0.125340\dots
\end{align*}
Here we have that $\kappa_{1}(2,2)=9$. 
Theorem~\ref{F_N} tells us that $\Phi_{1a}(2,2;\tau)^3$ and 
$\Phi_{1b}(2,2;\tau)^3$ are in $\mathcal{F}_9,$ so we may use 
Theorem~\ref{GaloisThm} to find the conjugates of the values 
of the functions at $\tau=\iup/3$. 
We have $\kappa_1(2,2)\cdot\iup/3=3\iup$ and 
\[
W_{9,3i}=\bigg\{\begin{pmatrix}t&0\\s&t\end{pmatrix} 
\in \GL_2(\Z/9\Z) \bigg\}/
\bigg\{\pm\begin{pmatrix}1&0\\0&1\end{pmatrix} \bigg\},
\] 
which has $27$ elements. However each of these acts like the identity 
on $\Phi_{1a}(2,2;\tau)$, and the group has an orbit of size three 
when acting on $\Phi_{1b}(2,2;\tau)$. 
The set $\mathcal Q_{36}$ has two elements 
$Q_1=x^2+9y^2$ and $Q_2=2x^2+2xy+5y^2$. 
These qive us $\beta_{Q_1}$ which is the identity, 
and $\beta_{Q_2}=\big(\begin{smallmatrix}2&1\\0&1\end{smallmatrix}\big)$. 
Therefore $\Phi_{1a}(2,2;\iup/3)^3$ has only one other conjugate, 
\[
\bigg(g_{2/9,1/9}\Big(\frac{-1+3\iup}{2}\Big)g_{4/9,2/9}
\Big(\frac{-1+3\iup}{2}\Big) g_{6/9,3/9}\Big(\frac{-1+3\iup}{2}\Big)
g_{8/9,4/9}\Big(\frac{-1+3\iup}{2}\Big)\bigg)^{-3},
\]
although the multiset described in Theorem~\ref{thm}~(2) contains $27$ 
copies of these two numbers. 
On the other hand, $\Phi_{1b}(2,2;\iup/3)^3$ has an orbit of six conjugates, 
and the multiset from Theorem~\ref{thm}~(2) contains nine copies of this orbit.
Theorem~\ref{thm2}~(2) tells us that $\Phi_{1a}(2,2;\iup/3)$ and 
$\Phi_{1b}(2,2;\iup/3)$ may have denominators which are powers of three,
whereas Theorem~\ref{thm2}~(1) tells us that their inverses are algebraic 
integers.
Therefore, we find the minimal polynomials for the inverses and then invert 
the polynomials. 
In this way, we find that $\Phi_{1a}(2,2;\iup/3)$ and 
$\Phi_{1b}(2,2;\iup/3)$ are roots of the irreducible polynomials
\begin{gather*}
3x^2-1\\
19683x^{18}-80919x^{12}+39366x^9+11016x^6+486x^3-1.
\end{gather*}
The full polynomials whose roots are the elements of the multisets 
corresponding to $\Phi_{1a}(2,2;\iup/3)^3$ and $\Phi_{1b}(2,2;\iup/3)^3$, 
counting multiplicity are 
\begin{gather*}
(27x^2-1)^{27}\\
(19683x^{6}-80919x^{4}-39366x^3+11016x^2-486x^2-1)^9.
\end{gather*}

Applying Theorem~\ref{thm2}(2), we find that $\sqrt{3}\Phi_{1a}(2,2;\iup/3)$ 
and $\sqrt{3}\Phi_{1b}(2,2;\iup/3)$ are units and roots of the polynomials 
\begin{gather*}
x-1\\
x^{18}+6 x^{15}-93 x^{12}-304 x^9+420 x^6-102 x^3+1.
\end{gather*}
Lastly, Theorem~\ref{thm2}~(3) applies, and we know that the ratio
\begin{align*}
\frac{\Phi_{1a}(2,2;\tau)}{\Phi_{1b}(2,2;\tau)}
&=q^{-2/3}\prod_{n=1}^{\infty}
\frac{(1-q^{9n-4})(1-q^{9n-5})}{(1-q^{9n-1})(1-q^{9n-8})}\\
&=q^{-2/3}(1+q+q^2+q^3-q^5-q^6-q^7+\cdots)
\end{align*}
evaluates to a unit at $\tau=\iup/3$. In fact we find that
\[
\frac{\Phi_{1a}(2,2;\iup/3)}{\Phi_{1b}(2,2;\iup/3)}=4.60627\dots
\]
is a unit. Indeed, it is a root of 
\[
x^{18}-102x^{15}+420x^{12}-304x^9-93x^6+6x^3+1.
\]
\end{example}

\begin{example} Here we give an example which illustrates the second 
remark after Theorem~\ref{thm2}. This is the discussion concerning ratios
of singular values of $\Phi_2$ and $\Phi_3$ with the same $\kappa_*$. 
Here we show that these ratios are not generically algebraic integral units
as Theorem~\ref{thm2}(3) guarantees for the $\A_{2n}^{(2)}$ cases.

We  consider $\Phi_2(1,1;\tau)$ and $\Phi_3(1,2;\tau)$, with 
$\tau=\sqrt{-1/3}$. 
For these example we have $\kappa_2(1,1)=\kappa_3(1,2)=6$. 
A short computation by way of the $q$-series shows that
\[
\Phi_2\big(1,1;\sqrt{-1/3}\big)=0.883210\dots,
\]
and
\[
\Phi_3\big(1,2;\sqrt{-1/3}\big)=0.347384\dots.
\]
Since $\Phi_2(1,1;\tau)^{24}$ and $\Phi_3(1,2;\tau)^{24}$ are in 
$\mathcal F_{12}$, we find that
\[
\Phi_2\big(1,1;\sqrt{-1/3}\big)^{24}\quad\text{and}\quad
\Phi_3\big(1,2;\sqrt{-1/3}\big)^{24}
\]
each have one other conjugate, namely 
\[
\Big(g_{1/2,1/3}\big(\sqrt{-4/3}\big)\cdot 
g_{1/4,0}\big(2\sqrt{-4/3}\big)\Big)^{-24} 
\quad\text{and}\quad
\Big(g_{0,1/3}\big(\sqrt{-4/3}\big)\cdot 
g_{1/2,0}\big(2\sqrt{-4/3}\big)\Big)^{-24}
\]
respectively, and the corresponding multisets described in 
Theorem~\ref{thm}~(2) each contain six copies of the respective orbits.
In this way we find that $\Phi_2\big(1,1;\sqrt{-1/3}\big)$ is a root of  
\[
2^{20}\,x^{48}-2^{12}\cdot 13x^{24}+1
\]
and $\Phi_3(1,2;\sqrt{-1/3})$ is a root of 
\[
2^{20} 3^{12}x^{48}-12^6 \cdot 35113x^{24}+1.
\]
Therefore, their ratio
\[
\frac{\Phi_2(1,1;\sqrt{-1/3})}{\Phi_3(1,2;\sqrt{-1/3})}=2.542459\dots
\]
is not a unit. Its minimal polynomial is 
\[
x^4-6x^2-3.
\]
\end{example}

\bibliographystyle{amsplain}

\end{document}